\documentclass[a4paper,11pt]{article}
\usepackage{geometry}
\usepackage[english]{babel}
\usepackage{authblk}
\author[1]{Wenjing Zhang}
\author[2]{Xiaowen Zhou}
\affil[1]{Laboratory of Mathematics and Complex Systems, School of Mathematical Sciences, Beijing Normal University, Beijing, China}
\affil[2]{Department of Mathematics and Statistics, Concordia University, Montreal, Canada}

\usepackage{amsfonts,amsmath,mathrsfs,amssymb}
\usepackage{ctex}
\usepackage{hyperref}
\usepackage{amsthm}
\usepackage{enumitem}
\usepackage{color}
\usepackage{bbm}

\def\beqnn{\begin{eqnarray*}}\def\eeqnn{\end{eqnarray*}}

\newtheorem{theorem}{Theorem}[section]% meant for sectionwise numbers
\newtheorem{proposition}[theorem]{Proposition}
\newtheorem{lemma}[theorem]{Lemma}
\newtheorem{coro}[theorem]{Corollary}% 

\newtheorem{example}{Example}%
\newtheorem{remark}{Remark}%

\newtheorem{condition}{Condition}%
\numberwithin{equation}{section}
\newcommand{\dd}{\mathrm{d}}

\def\benumerate{\begin{enumerate}}\def\eenumerate{\end{enumerate}}
\def\bitemize{\begin{itemize}}\def\eitemize{\end{itemize}}

\def\beqlb{\begin{eqnarray}}\def\eeqlb{\end{eqnarray}}
\def\beqnn{\begin{eqnarray*}}\def\eeqnn{\end{eqnarray*}}

\def\<{\langle}\def\>{\rangle}

\begin{document}
	
	\title{Propagation of support for super-Brownian motion with general branching mechanism}
	\maketitle
\abstract{
We study the spatial propagation of super-Brownian motion on $\mathbb{R}^d$ with general critical or subcritical (spatially dependent) branching mechanisms. Under local spatial lower bounds satisfying a Keller-Osserman type integrability condition, we establish a quantitative upper bound for the short-time probability that the support exits a prescribed neighborhood of its initial support. The estimate has a Gaussian-tail form and is obtained through weighted occupation times, Feynman–Kac representations, singular elliptic boundary blow-up estimates, and mild comparison arguments for log-Laplace equations. As an application, we derive the compact support property directly for spatially dependent branching mechanisms satisfying suitable local lower bounds. This yields a sufficient compact-support criterion expressed in terms of the inverse Keller integral. In particular, for  spatially dependent super-Brownian motions with stable branching,  we generalize the compact support results of Engländer--Pinsky and Ren.

   }
 \vskip 0.5cm
 Keywords: super-Brownian motion, general branching mechanism, spatially dependent branching, compact support property, Keller integral.  
\vskip 0.5cm

 AMS 2020 classification: 60J68, 60G57.
 
	\section{Introduction}\label{sec1}
	\indent %In the theory of stochastic processes, 
    Dawson-Watanabe superprocesses represent an important class of measure-valued branching processes, systematically studied by Watanabe \cite{Wata68}, Dawson \cite{Dawson75}, among others, beginning in the late 1960s. Its study originated from in-depth investigations into the scaling limit  of branching particle systems and the need for modeling random evolution processes with spatial structure. One classical model is the super-Brownian motion, where the spatial motion is a Brownian motion. Classical studies often assume a quadratic (i.e. binary branching) or stable branching mechanism, and such models have been extensively studied in terms of support properties and extinction behaviors. However, for general non-self-similar branching mechanisms, especially in spatially dependent settings, explicit quantitative upper bounds on short-time propagation of support remain much less developed.

	The foundational works by Dawson \cite{Dawson75}, Watanabe \cite{Wata68}, and others established the existence and basic properties of superprocesses. Subsequent research has extensively explored their path properties, including the compact support property and propagation rate of the support. In the case of the binary branching mechanism, Iscoe \cite{Iscoe86} made a seminal contribution by introducing the method of weighted occupation time. This powerful technique provided a framework for analyzing the spatial spread of the superprocess and was instrumental in establishing precise estimates for the growth of the support, which, along with the probabilistic potential theory approach developed by Dynkin \cite{Dynkin1991PDE, MR1280712}, formed the foundation for much of the later analysis on the super-Brownian motion. Dawson, Iscoe, and Perkins \cite{DIP89} studied the path properties for super-Brownian motion systematically and established key sample path properties of super-Brownian motion, including a L\'evy-type modulus of continuity for the support, exact asymptotics for hitting probabilities, the exact Hausdorff measure of the range, and sharp criteria for the polarity and dimension of multiple points. For super-Brownian motions driven by stable and spatially constant branching mechanisms, the propagation speed and support properties are well-understood. In particular, due to the  self-similarity and scaling properties inherent in stable branching, Dawson and Vinogradov \cite{dawson1994almost} obtained rather sharp   upper and lower bounds for the speed of the support propogation in a short time. Interestingly, Dawson and Vinogradov explicitly pointed out   that their propagation results might potentially be carried over to more general branching mechanisms under certain restrictions; see \cite[Remark 2.1]{dawson1994almost}. Later, Dhersin and Le Gall \cite{DhersinLeGall98} established the Kolmogorov test for the super-Brownian motion with binary branching mechanism. This integral test gives precise information on the speed of  super-Brownian motion started from a Dirac mass moving away from its starting point, which refines the previous results due to Dawson and Vinogradov \cite{dawson1994almost}.
    
    %The core difficulty lies in the fact that the 
    The analysis becomes  more involved  when the branching mechanism lacks self-similarity or depends  on spatial position. In the absence of exact scaling properties, determining precise asymptotics for the macroscopic propagation speed remains  challenging. Although fine microscopic path properties—such as the Hausdorff dimension of the support and hitting probabilities of microscopic small balls—have been analyzed for general mechanisms using the Brownian snake (see, e.g., Delmas \cite{Delmas99}), explicit quantitative estimates for the macroscopic speed of propagation remain elusive. Consequently, much of the existing literature on super-Brownian motion with general and spatially dependent settings has focused on qualitative support properties rather than explicit quantitative speed estimates. For example, Sheu \cite{Sheu94,SHEU1997129} provided definitive criteria for the compactness of the support. Later, Engländer and Pinsky \cite{JP99,JP06} related this compact support property to qualitative uniqueness criteria for the associated semi-linear equations, and Ren \cite{Ren04} systematically investigated the support properties of critical super-Brownian motions equipped with spatially dependent branching rates. Hesse and Kyprianou \cite{HK14} studied the mass of a super-Brownian motion as it first exits an increasing sequence of balls, which constitutes a time-inhomogeneous continuous-state branching process. In this perspective,  Sheu's compact-support condition is  transformed into Grey's condition for a limiting continuous-state branching process.

In this paper, we  generalize the result of Dawson and Vinogradov \cite{dawson1994almost} to superprocesses whose branching mechanism admits spatially dependent local lower bounds. Rather than seeking a matching pair of upper and lower bounds, which rely heavily on self-similarity, our primary result establishes a short-time propagation estimate under the local lower bound Condition \ref{con_localize}. Specifically, for an initial mass $\mu$ with compact support $S(\mu)$, our main result (Theorem \ref{th1}) shows that the probability of the process propagating beyond an $R$-neighborhood of $S(\mu)$ before a short time $t$ admits a Gaussian-tail bound with a multiplier   controlled by the inverse Keller integral associated with the local lower bound branching mechanism.
This explicit speed estimate naturally implies the qualitative compact support property as a straightforward consequence. We formulate the compact-support criterion directly for mechanisms with local constant lower bounds (Theorem \ref{th:csp_unify}); the global uniform lower bound case (Corollary \ref{coro2}) and the spatially constant case are then recovered as special cases. This work thus extends the classical results for binary and stable branching to a much wider class of models, enhancing our understanding of superprocesses in non-homogeneous environments. The proof of Corollary \ref{coro:regularVarying} in Section
\ref{sec:proof_regular_varying} shows that the criterion is not restricted
to exact polynomial branching mechanisms.

The remainder of this paper is structured as follows. In Section \ref{Preliminaries} we introduce  super-Brownian motions with general branching mechanisms and state the assumptions used throughout the paper. In Section \ref{mainre} we  state the main results. Theorem \ref{th1} gives a quantitative upper bound for the short-time propagation probability, and Theorem \ref{th:csp_unify} and Corollary \ref{coro2} derive compact-support consequences for spatially dependent branching mechanisms. Section \ref{sec2} proves the propagation estimate using weighted occupation
times, Feynman--Kac representations, and mild comparison arguments for
log-Laplace equations. Section \ref{sec4} applies the estimate to compact-support properties and proves the regularly varying criterion stated in Corollary \ref{coro:regularVarying}; the stable-type specialization and its comparison with the classical criteria are discussed immediately after Corollary \ref{coro:regularVarying}.

\section{Preliminaries}\label{Preliminaries}
In this paper, we focus on the support property of super-Brownian motions on $\mathbb{R}^d$ , $d\geq 1$, with general branching mechanisms. We will first define general branching mechanisms and give an explicit construction and characterization of the corresponding superprocesses.

By a general branching mechanism, we mean a function $\Psi$ defined on $\mathbb{R}^d\times[0,\infty)$ of the following form
$$\Psi(x,u)=\alpha(x)u+\beta(x) u^2+\int_0^\infty (e^{-\lambda u}-1+\lambda u)\pi(x,\mathrm{d}\lambda),\quad x\in\mathbb{R}^d, u\geq0$$
where  $\alpha(\cdot)$ and $\beta(\cdot)$ are bounded continuous functions on $\mathbb{R}^d$, $\beta$ is nonnegative and $(\lambda\land \lambda^2)\pi(x,\mathrm{d}\lambda)$ is a bounded kernel from $\mathbb{R}^d$ to $(0,\infty)$. 
	We say that $\Psi$ is \textit{spatially constant} if $\alpha(x), \beta(x)$ and $\pi(x,\mathrm{d}\lambda)$ do not depend on the spatial position $x$. For ease of notation, we will always use $\Phi$ to denote a spatially constant branching mechanism of the form
	$$\Phi(u)=\alpha u+\beta u^2+\int_0^\infty (e^{-\lambda u}-1+\lambda u)\pi(\mathrm{d}\lambda),\quad u\geq0,$$
	where $\alpha\in\mathbb{R},\beta\geq0$, $\pi(\mathrm{d}\lambda)$ is a $\sigma$-finite measure on $(0,\infty)$ that satisfies
	$$\int_{(0,\infty)}  (\lambda\land \lambda^2)\pi(\mathrm{d}\lambda)<\infty.$$
	Otherwise, we say that the branching mechanism $\Psi$ is \textit{spatially dependent}. We say that the branching mechanism $\Psi$ is \textit{subcritical} if $\alpha(x)> 0, \forall x\in\mathbb{R}^d$ (or $\alpha>0$ for spatially constant branching), and \textit{critical} if $\alpha(x)\equiv0$ (or $\alpha=0$ for spatially constant branching). 
    
Throughout the paper, degenerate branching mechanisms are excluded. For a (sub)critical spatially constant mechanism this means
$$
\alpha+\beta+\pi((0,\infty))>0.
$$
For a spatially dependent mechanism $\Psi$, this means that, unless otherwise
specified,
$$
\alpha(x)+\beta(x)+\pi(x,(0,\infty))>0,
\qquad x\in\mathbb{R}^d.
$$
Throughout the paper, unless explicitly stated otherwise, all branching mechanisms under consideration are critical or subcritical. For $u>0$, the L\'evy--Khintchine representation gives
$$
\partial_u\Psi(x,u)
=
\alpha(x)+2\beta(x)u+
\int_0^\infty \lambda(1-e^{-\lambda u})\pi(x,\dd\lambda).
$$
Consequently, under the critical or subcritical assumption, $\Psi(x,\cdot)$ is non-decreasing on $[0,\infty)$, and it is
strictly increasing on $[0,\infty)$ under the above non-degeneracy condition.
    
	\subsection{Super-Brownian motion with general branching mechanism}
    Let $M_F=M_F(\mathbb{R}^d)$ be the space of finite measures on $(\mathbb{R}^d,\mathscr{B}(\mathbb{R}^d))$ furnished with the weak topology. For any measure $\mu\in M_F$, we use $S(\mu)$ to denote its topological support. We  use $M_c$ to denote the subspace of $M_F$ consisting of measures with compact support. We consider a $d$-dimensional standard Brownian motion $W=(\Omega,\mathscr{F},(\mathscr{F}_t)_{t\geq0},(W_t)_{t\geq0},P_x)$, and a general branching mechanism $\Psi$ given as before. Then, we can construct a time-homogeneous $M_F$-valued branching Markov process with spatial motion $W$ and branching mechanism $\Psi$, called a \textit{super-Brownian motion with a general branching mechanism} $\Psi$ (on $\mathbb{R}^d$). The existence of such superprocesses is well known. In the following, we will use $X=(\mathbb{D},\mathscr{D},(\mathscr{D}_{t+})_{t\geq 0}, (X_t)_{t\geq 0},(\mathbb{P}_\mu)_{\mu\in M_F})$ to denote the canonical Borel right realization of super-Brownian motion with a  general branching mechanism. Here $\mathbb{D}=\mathbb{D}([0,\infty),M_F)$ is the Skorokhod space, i.e., the family of all c\`{a}dl\`{a}g mappings from $[0,\infty)$ to $M_F(\mathbb{R}^d)$, which is endowed with the Skorokhod topology (see Ethier and Kurtz \cite{MR838085} for details of the theory of Skorokhod space). The process is defined as a map $X:\mathbb{D}\to M_F(\mathbb{R}^d)$ with $X_t(\omega)=\omega(t)$, $\mathscr{D}_t=\sigma\{X_s:0\leq s\leq t\}$, $\mathscr{D}=\bigvee \mathscr{D}_t$ and $\mathbb{P}_\mu$ is a probability measure on $(\mathbb{D},\mathscr{D})$ with its Laplace transform  satisfying
	\begin{equation}\label{cumu}
		\mathbb{P}_\mu(\exp\left<-f,X_t\right>)=\exp\left<-(V_tf),\mu\right>,
	\end{equation}
	where $f$ belongs to the space $\text{bp}\mathscr{B}(\mathbb{R}^d)$ of bounded measurable non-negative functions on $\mathbb{R}^d$. Here and in the following, we use $\left<v,\mu\right>$ or $\mu(v)$ to denote the integral of a measurable function $v$ with respect to a measure $\mu$. In the above equation, $(V_t)$ is the so-called cumulant semigroup of $X$.
	For any given $f$, $v(x,t):=(V_tf)(x)$ is characterized as the unique non-negative solution of the integral equation
	\begin{equation}\label{inteq}
	v(x,t)=P_x [f(W_t)]-P_x\left[\int_0^t \Psi(W_s,v(W_s,t-s))\mathrm{d}s\right].
	\end{equation}
    
	The integral equation \eqref{inteq} is the mild formulation of the formal semilinear equation
	\begin{equation}\label{deq}
		\left\{
		\begin{aligned}
			\dfrac{\partial v(x,t)}{\partial t}&= \dfrac{1}{2}\Delta v(x,t)-\Psi(x,v(x,t)),\quad t\geq 0, x\in\mathbb{R}^d;\\
			v(x,0)&=f(x),
		\end{aligned}
		\right.
	\end{equation}
	where $\Delta$ denotes the Laplacian, that is,
	$$\Delta:=\sum_{i=1}^d\dfrac{\partial^2}{\partial x_i^2}.$$

    For general theory of measure-valued branching processes and their deep connections with non-linear partial differential equations, one can also refer to Dawson \cite{Dawson93}, Dynkin \cite{MR1280712}, \cite{LeGall99}, and Perkins \cite{MR1915445}, where the existence and uniqueness of the mild solutions to the integral equation \eqref{inteq} can be found. In particular, we adopt the notation and formulation in Li \cite{Li22}.

	%\paragraph{Weighted Occupation time}
    The weighted occupation time method, first introduced by Iscoe  \cite{Iscoe86}, provides a powerful framework for analyzing the spatial spread of superprocesses. For a super-Brownian motion $X$, the weighted occupation time up to time $t$ is defined as a random measure $Y_t(\cdot)$, so that for any bounded measurable function $f: \mathbb{R}^d \to [0,\infty)$, 
	\begin{equation}
		Y_t(f) := \int_0^t \langle f, X_s \rangle ds.
	\end{equation}
	This functional captures the cumulative mass of the process weighted by the spatial function $f$ over time. A key feature is that the Laplace transform of $Y_t(f)$ is represented by the unique non-negative mild solution of an integral equation. Specifically, for $\theta > 0$, we have
    \begin{equation}\label{eq:weighted_occupation_laplace}
		\mathbb{P}_\mu\left[\exp\left(-\theta Y_t(f)\right)\right] = \exp\left(-\langle u_\theta(\cdot, t), \mu\rangle\right),
	\end{equation}
	where $u_\theta(x,t)$ is the unique non-negative mild solution, locally bounded on each finite time interval, of the following integral equation
\begin{equation}\label{eq:weighted_occupation_mild}
u_\theta(x,t)=P_x\left[\int_0^t\left\{
\theta f(W_s)-\Psi(W_s,u_\theta(W_s,t-s))\right\}\dd s\right].
\end{equation}
The formal
differential form of \eqref{eq:weighted_occupation_mild} is
\begin{equation}\label{eq:weighted_occupation_formal}
		\frac{\partial u_\theta}{\partial t} = \frac{1}{2}\Delta u_\theta - \Psi(x,u_\theta) + \theta f(x), \quad u_\theta(x,0) = 0.
	\end{equation}

    This representation allows one to translate probabilistic problems about the spatial spread of the superprocess into analytic problems concerning the solutions of these mild log-Laplace equations, providing a powerful tool for obtaining precise estimates on propagation speeds and support properties.
	We refer the reader to Iscoe \cite{Iscoe86} and Li \cite{Li22} for further details on weighted occupation times. 
    \subsection{Conditions on branching mechanisms}
    In this subsection, we impose several conditions on the general branching mechanisms. We first focus on the spatially constant branching mechanism denoted by $\Phi(u)$.
    \begin{condition}\label{con1}
			The branching mechanism $\Phi$ is (sub)critical, and satisfies the integral condition
			\begin{equation}\label{intcon}
				\int_\theta^\infty\left(\int_0^\lambda \Phi(u)\mathrm{d}u\right)^{-1/2}\mathrm{d}\lambda<\infty,\quad \textrm{for some } \theta>0.
			\end{equation}
\end{condition}
By the non-degeneracy convention above, every critical or subcritical
branching mechanism considered in Condition \ref{con1} is strictly increasing
on $(0,\infty)$. This monotonicity will help us in the application of comparison principles. Sheu \cite{Sheu94,SHEU1997129} proved that Condition \ref{con1} ensures that the super-Brownian motion with the branching mechanism $\Phi$ has the compact support property.  Specifically, it is easy to verify that the binary branching mechanism $\Phi(u)=\beta u^2$ and the stable branching mechanism $\Phi(u)=\beta u^{1+p}$ for $p\in(0,1)$ satisfy Condition \ref{con1}.

We formulate another important property, which is closely related to the criterion for extinction of continuous-state branching processes.

\begin{condition}[Grey's condition] \label{GreyCon}
The branching mechanism $\Phi$ is (sub)critical and satisfies the integral condition
\begin{equation}\label{Grey}
\int_\theta^\infty \dfrac{\mathrm{d}u}{\Phi(u)}<\infty,\quad \text{for some } \theta>0.
\end{equation}
\end{condition}
It is well known that Grey's condition \eqref{Grey} holds if and only if the $\Phi$-continuous state branching process becomes extinct with positive probability; see Grey \cite{Grey_1974} for further details. We also note that, if we further assume that the branching mechanism $\Phi$ is (sub)critical, then both the $\Phi$-continuous state branching process and the super-Brownian motion started from a finite measure $\mu$ with the branching mechanism $\Phi$ become extinct with probability $1$. 

It is easy to see that  Condition \ref{con1} implies Condition \ref{GreyCon}; see \cite{SHEU1997129}. 

Next, we turn to the spatially dependent case. We assume that the spatially dependent branching mechanism is locally bounded below by a family of spatially constant branching mechanisms in bounded domains. To formalize this extension, let $\Psi(x,u)$ be a general spatially dependent branching mechanism defined on $\mathbb{R}^d\times [0,\infty)$. We will always use $B(x,R)$ to denote the open ball centered at $x$ with radius $R$, and $\overline{B(x,R)}$ to denote the corresponding closed ball. We impose the following condition.
\begin{condition}\label{con_localize}
For any $R>0$, there exists a spatially constant (sub)critical branching mechanism $\Phi_R$ satisfying Condition \ref{con1}, such that
$$\Phi_R(u)\leq\Psi(x,u),\quad \forall x\in \overline{B(0,R)}, u\geq 0.$$
\end{condition}
A particular case is when the spatially 
dependent branching mechanism $\Psi$ is uniformly bounded below by a spatially constant branching mechanism $\Phi$:
\begin{condition}\label{con_general}
The branching mechanism $\Psi$ is bounded below by a (sub)critical spatially constant branching mechanism $\Phi(u)$, i.e.
\begin{equation}
\Phi(u)\leq \Psi(x,u),\quad \forall x\in\mathbb{R}^d, u\geq 0.
\end{equation}
and $\Phi(u)$ satisfies Condition \ref{con1}.
\end{condition}
The conditions introduced above will be used frequently in the statement of our results. Condition \ref{con1} is used to construct and estimate the boundary blow-up solution. Condition \ref{GreyCon} is used only for extinction. Conditions \ref{con_localize} and \ref{con_general} allow one to estimate super-Brownian motion  with spatially dependent branching mechanism using that with spatially dependent mechanism for which more explicit results are available. The relation between them is
$$
\text{Condition }\ref{con1}\implies\text{Condition }\ref{GreyCon},
\qquad
\text{Condition }\ref{con_general}\implies\text{Condition }\ref{con_localize}.
$$
Condition \ref{con_localize} involves a family of spatially constant lower bounds $\{\Phi_R\}_{R>0}$ satisfying Condition \ref{con1}, whereas Condition \ref{con_general} is the special case in which one may take $\Phi_R=\Phi$ for all $R>0$.

\subsection{The Keller integral}\label{subsec_Keller}
While Condition \ref{GreyCon} gives the criterion of extinction, Condition \ref{con1} serves as the criterion of existence of solutions to certain singular boundary value problem related to the superprocesses. 

Precisely, suppose that a spatially constant branching mechanism $\Phi$ satisfying Condition \ref{con1} is given. For any $R>0$ we consider the following singular boundary value problem
\begin{equation}\label{SBVP}
\left\{\begin{aligned}
\dfrac{1}{2}\Delta u_{R,\Phi}(x)&=\Phi(u_{R,\Phi}(x)),\quad x\in B(0,R); \\
u_{R,\Phi}(x)&\to+\infty,\quad as\ |x|\uparrow R.
\end{aligned}
\right.
\end{equation}

The study of such singular boundary value problems has a long history, with foundational work dating back to the seminal papers of Keller \cite{Keller57} and Osserman \cite{Osserman57}, who established the basic theory of boundary blow-up for nonlinear elliptic equations. In addition, they established the existence of a solution of \eqref{SBVP} under Condition \ref{con1}. In particular, under Condition \ref{con1}, it follows from Costin and Dupaigne \cite[Corollary 1.4]{CD10} that the solution of \eqref{SBVP} is unique. In addition, Sheu \cite[Lemma 2.3(C) and Theorem 2.3]{Sheu94}  showed that this unique solution of equation \eqref{SBVP} can be constructed through the logarithmic Laplace functional of the super-Brownian motion stopped when hitting the boundary $\partial B(0,R)$. Those results imply that under Condition \ref{con1}, $\eqref{SBVP}$ admits a radial solution, and we denote it by $u_{R,\Phi}$. For $r\in[0,R)$, we write $u_{R,\Phi}(r)$ for its common value on $\{x\in\mathbb{R}^d:\|x\|=r\}$.

In particular, following the work of Keller \cite{Keller57}, we can derive an upper-bound estimate of the solution of \eqref{SBVP}. First, assume that we are given a branching mechanism $\Phi$ satisfying Condition \ref{con1}.
Define  the Keller integral 
$$G_{\Phi}(z):=\int_z^\infty\dfrac{\dd s}{\sqrt{2\int_0^s \Phi(v)\dd v}},\quad z>0.$$
 Since $\Phi$ satisfies Condition \ref{con1}, $G_{\Phi}$ is finite and strictly decreasing. Let $G_{\Phi}^{-1}$ be its inverse function. It is easy to see that $G_{\Phi}^{-1}$ is a mapping from $(0,\infty)$ to $(0,\infty)$ with 
 $$
 \lim_{s\to0 }G_{\Phi}^{-1}(s)=\infty,\quad \lim_{s\to\infty}G_{\Phi}^{-1}(s)=0.
 $$
 From the proof in Keller \cite[Theorem I]{Keller57}, for any $R>|x|>0$,
\begin{align}\label{eq:inverseUpperBound}
  u_{R,\Phi}(x) \le G_{\Phi}^{-1}(R - |x|).  
\end{align}
\begin{example}\label{ex:keller_stable}
If the branching mechanism $\Phi$ is stable, that is,
$$
\Phi(u)=\beta u^{1+p},\quad u\geq 0,
$$
for some constants $\beta>0$ and $p\in(0,1]$, then $\Phi$ satisfies Condition
\ref{con1}. Indeed,
$$\int_0^s\Phi(v)\dd v=\frac{\beta}{p+2}s^{p+2}.$$
Therefore, for every $\theta>0$,
\begin{align*}
\int_\theta^\infty
\dfrac{\dd s}{\sqrt{\int_0^s\Phi(v)\dd v}}=\int_\theta^\infty\dfrac{\dd s}{\sqrt{\frac{\beta}{p+2}s^{p+2}}}=\sqrt{\frac{p+2}{\beta}}\int_\theta^\infty s^{-1-p/2}\dd s=\frac{2}{p}\sqrt{\frac{p+2}{\beta}}\theta^{-p/2}<\infty.
\end{align*}
Moreover, the Keller integral is given explicitly for $z>0$ by
\begin{align*}
G_{\Phi}(z)&=\int_z^\infty\dfrac{\dd s}{\sqrt{2\int_0^s\Phi(v)\dd v}}=\int_z^\infty\dfrac{\dd s}{\sqrt{\frac{2\beta}{p+2}s^{p+2}}}=\sqrt{\frac{p+2}{2\beta}}\int_z^\infty s^{-1-p/2}\dd s=\frac{2}{p}\sqrt{\frac{p+2}{2\beta}}z^{-p/2}.
\end{align*}
Consequently, if
$$A_{p,\beta}:=\frac{2}{p}\sqrt{\frac{p+2}{2\beta}},$$
then
$$G_{\Phi}(z)=A_{p,\beta}z^{-p/2},\quad z>0.$$
Solving $r=A_{p,\beta}z^{-p/2}$ for $z$, we obtain
$$
G_{\Phi}^{-1}(r)= A_{p,\beta}^{2/p}r^{-2/p}=
\left(\frac{2}{p}\sqrt{\frac{p+2}{2\beta}}\right)^{2/p}
r^{-2/p},
\quad r>0.
$$
\end{example}

\section{Main results}\label{mainre}
In this section, we state our main results on the propagation speed estimate of a super-Brownian motion with general branching mechanisms and give some direct corollaries.

Our first main result concerns  the propagation speed of the superprocess support. For a super-Brownian motion $X$ with a branching mechanism $\Psi$ satisfying the local lower bound Condition \ref{con_localize}, we give an upper estimate of the probability of the event that $X$ visits the complement of the $R$-neighborhood of the support of its initial measure during the time period $[0,t]$. The global lower bound and spatially constant cases are recovered by taking the lower bound family to be independent of the radius, and in particular this generalizes \cite[Corollary 1.5]{dawson1994almost}, in which the $d$-dimensional super-Brownian motions with spatially constant stable branching mechanism are considered.

{For any $A\subset\mathbb{R}^d$ and $\varepsilon>0$, denote by
$$A^\varepsilon:=\{x\in\mathbb{R}^d:\operatorname{dist} (x,A)\leq \varepsilon\}
$$
the $\varepsilon$-neighborhood of the set $A$, where here $\operatorname{dist}$ denotes the distance function in the Euclidean space.} For each $\mu\in M_c$, define
\begin{equation}\label{defrho}
\rho(\mu):=\inf\{r>0: S(\mu)\subset B(0,r)\}.
\end{equation}
 \begin{theorem}\label{th1}
Given a super-Brownian motion $X$ with a branching mechanism $\Psi$ satisfying Condition \ref{con_localize}, for any $\mu\in M_c$, $0<R_1<R$ and $0<t<\dfrac{R_1^2}{d}$, we have
$$
\begin{aligned}
&\mathbb{P}_\mu\left\{X_u\left((S(\mu)^R)^c\right) > 0 \text{ for some } u \leq t\right\} \\
&\quad \leq \mu(1) \cdot u_{R,\Phi_{\rho(\mu)+R}}(R_1) \cdot C(d) \left(\frac{R_1}{\sqrt{t}}\right)^{d-2} \exp\left(-\frac{R_1^2}{2t}\right) \\
&\quad \leq \mu(1) \cdot G_{\Phi_{\rho(\mu)+R}}^{-1}(R-R_1) \cdot C(d) \left(\frac{R_1}{\sqrt{t}}\right)^{d-2} \exp\left(-\frac{R_1^2}{2t}\right),
\end{aligned}
$$
where $C(d)$ is a constant depending only on dimension $d$, and $u_{R,\Phi_{\rho(\mu)+R}}$ denotes the unique radial classical solution of singular boundary value problem \eqref{SBVP} in $B(0,R)$ with $\Phi$ replaced by $\Phi_{\rho(\mu)+R}$. If $\Psi$ satisfies Condition \ref{con_general} with lower bound $\Phi$, then this estimate reduces to the same bound with $u_{R,\Phi_{\rho(\mu)+R}}$ and $G_{\Phi_{\rho(\mu)+R}}^{-1}$ replaced by $u_{R,\Phi}$ and $G_\Phi^{-1}$, respectively. In particular, the spatially constant case is obtained by taking $\Psi(x,u)=\Phi(u)$.
\end{theorem}
The proof of Theorem \ref{th1} is given in Subsection \ref{subsec:proof_th1}.

\begin{remark}
For spatially constant super-Brownian motion with branching mechanism $\Phi(u)=\beta u^{1+p}$, the Keller integral is given by Example \ref{ex:keller_stable}, and by Theorem \ref{th1} the upper bound of the propagation probability is given explicitly by
\begin{align}
\mathbb{P}_\mu\left\{X_u\left((S(\mu)^R)^c\right) > 0 \text{ for some } u \leq t\right\} \le \mu(1) \cdot C_{p,\beta} \cdot (R-R_1)^{-2/p} \left(\frac{R_1}{\sqrt{t}}\right)^{d-2} \exp\left(-\frac{R_1^2}{2t}\right).
\end{align}
Optimizing the upper bound, we have  for small $t>0$,
\begin{align}
 \mathbb{P}_\mu\left\{X_u\left((S(\mu)^R)^c\right) > 0 \text{ for some } u \leq t\right\} \le \widetilde{C}(p,d,\beta) \cdot \mu(1) \cdot \left(\frac{R}{t}\right)^{2/p} \left(\frac{R}{\sqrt{t}}\right)^{d-2} \exp\left(-\frac{R^2}{2t}\right).
\end{align}
Note that the above expression has exactly the same decay rate as that in Dawson and Vinogradov \cite[Corollary 1.5]{dawson1994almost}. Hence, Theorem \ref{th1} generalizes the speed estimate given by Dawson and Vinogradov.
\end{remark}

After establishing this upper-bound estimate, we can apply this quantitative tool to investigate the compact support property, particularly extending it to processes with spatially dependent branching mechanisms. 

We define the range of a superprocess $X$ as the smallest closed set $\mathscr{R}(X)$ in $\mathbb{R}^d$ such that $S(X_t)\subset \mathscr{R}(X)$ for every $t\geq0$. {Formally, the range of $X$ can be expressed by
			$$\mathscr{R}(X):=\overline{\bigcup_{t\geq 0}S(X_t)}$$
			where $\overline{A}$ denotes the closure of $A\in\mathscr{B}(\mathbb{R}^d)$. }
Let $M_c=M_c(\mathbb{R}^d)$ be the family of finite measures on $\mathbb{R}^d$ with compact support. We say that $X$ exhibits the \textit{compact support property} if for any finite initial measure $\mu\in M_c(\mathbb{R}^d)$, and $t>0$, we have
		$$\mathbb{P}_\mu\left\{\overline{\bigcup_{s\leq t}S(X_s)} \text{ is compact}\right\}=1.$$
		
Historically, the compact support property for super-Brownian motion with a spatially constant branching mechanism was thoroughly investigated by Sheu \cite{Sheu94,SHEU1997129}. For spatially dependent settings, Engländer and Pinsky \cite{JP99,JP06} made significant progress by translating the compact support property into qualitative uniqueness criteria for non-negative solutions of associated semi-linear PDEs via $h$-transform techniques. 

However, earlier approaches often required the assumption of specific polynomial forms (such as $\alpha(x)u+\beta(x)u^{1+p}$) for spatially dependent mechanisms. Using the propagation estimates obtained from weighted occupation times and local parabolic barriers, we can provide a direct alternative verification of this property. More importantly, the method does not rely on exact spatial independence. We formulate the compact-support criterion directly for spatially dependent mechanisms with local lower bounds. The case of a global uniform lower bound, and in particular, the spatially constant case, is then recovered as special cases. Assume that the branching mechanism $\Psi$ satisfies Condition \ref{con_localize}. Under this condition, the propagation speed estimate still holds; its proof uses the spatially dependent local barrier collected in the Appendix \ref{app_1}. As a consequence, we obtain a sufficient condition for the compact support property.
 \begin{theorem}\label{th:csp_unify}
Let $X$ be a super-Brownian motion with a spatially dependent branching mechanism $\Psi$ satisfying Condition \ref{con_localize} with a local lower bound family $\{\Phi_R\}$ of branching mechanisms satisfying Condition \ref{con1}. A sufficient condition for $X$ to possess the compact support property in finite time, i.e., for any $t>0$ and $\mu\in M_c$,  
\begin{align}
\mathbb{P}_\mu\left\{ \overline{\bigcup_{s \le t} S(X_s)} \text{ is compact} \right\} = 1,
\end{align}
is given by
\begin{align}\label{eq:sufficientCon}
\liminf_{R\to\infty}\dfrac{\ln G_{\Phi_R}^{-1}(R/2)}{R^2}<\infty.
\end{align}
Equivalently, along some sequence $R_n\to\infty$, the quantity $G_{\Phi_{R_n}}^{-1}(R_n/2)$ grows at most like $\exp(C R_n^2)$ for some finite $C$.
\end{theorem}
The proof of Theorem \ref{th:csp_unify} is given at the beginning of Section \ref{sec4}.
\begin{remark}The appearance of $G_{\Phi_R}^{-1}$ in the above sufficient condition is due to $u_{R,\Phi_R}(R/2)\leq G_{\Phi_R}^{-1}(R/2)$ for large $R$.
\end{remark}

\begin{coro}\label{coro:regularVarying}
Let $X$ be a super-Brownian motion with  spatially dependent (sub)critical branching mechanism $\Psi$ satisfying \begin{equation}\label{eq:rv_lower_bound}
\Psi(x,u)\geq \beta(x)\Phi(u),
\qquad x\in\mathbb{R}^d,\quad u\geq0,
\end{equation}
where $\beta$ is a positive bounded continuous function and $\Phi$ is
regularly varying at infinity with index $1+p$ for $p\in(0,1]$, that is,
$$
\Phi(u)=u^{1+p}L(u),
\qquad u\to\infty,
$$
for a slowly varying function $L$ at infinity.
 Set $\beta_R:=\inf_{x\in\overline{B(0,R)}}\beta(x).$
  Then $X$ has the finite-time compact support property provided that
\begin{equation}\label{eq:rv_gaussian_condition}
\limsup_{R\to\infty}\frac{\ln \beta_R}{R^2}>-\infty.
\end{equation}
%Let $\Phi$ be a spatially constant (sub)critical branching mechanism satisfying
%Condition \ref{con1}. Assume that $\Phi$ is regularly varying at infinity with
%index $1+p$ for $p\in(0,1]$, that is,
%$$
%\Phi(u)=u^{1+p}L(u),
%\qquad u\to\infty,
%$$
%where $L$ is slowly varying at infinity. Consider a spatially dependent
%branching mechanism $\Psi$ satisfying
\end{coro}
The proof of Corollary \ref{coro:regularVarying} is given in Subsection \ref{sec:proof_regular_varying}. This 
corollary yields the compact support property
whenever the local minimum branching rate does not decay faster than Gaussian
order along some unbounded sequence of radii.

 In the following, we compare condition \eqref{eq:rv_gaussian_condition} with   
the 
compact-support criteria of Engl\"{a}nder--Pinsky \cite{JP99,JP06} and Ren
\cite{Ren04} for super-Brownian motions with spatially dependent stable branching mechanisms, that is, 
$$\Psi(x,u)=\beta(x)u^{1+p},\quad u\geq 0,x\in\mathbb{R}^d$$
where $\beta(x)>0$ is continuous and bounded, and $p\in(0,1]$.
Their results are formulated through qualitative uniqueness criteria for associated semilinear PDEs and, in regular model classes, identify the
Gaussian spatial-decay threshold. More precisely, a sufficient condition that appears in 
Engl\"ander--Pinsky \cite{JP99,JP06} and Ren \cite{Ren04} is the following. 
\begin{align}\label{con:ren}
\text{for some
constant } M>0, \beta(x)\geq \exp(-M|x|^2) \text{for all sufficiently large } |x|.
\end{align}

We first show that \eqref{con:ren} implies the local-minimum condition \eqref{eq:rv_gaussian_condition}. Choose $R_0>0$ such that the above bound holds
whenever $|x|\geq R_0$. Since $\beta$ is positive and continuous,
$$
m_0:=\inf_{|x|\leq R_0}\beta(x)>0.
$$
For every $R\geq R_0$ and every $x\in\overline{B(0,R)}$, either
$|x|\leq R_0$, in which case $\beta(x)\geq m_0$, or
$R_0<|x|\leq R$, in which case
$$
\beta(x)\geq \exp(-M|x|^2)\geq \exp(-MR^2).
$$
Consequently,
$$
\beta_R
=
\inf_{x\in\overline{B(0,R)}}\beta(x)
\geq
\min\{m_0,\exp(-MR^2)\}.
$$
For all sufficiently large $R$, we have $\exp(-MR^2)\leq m_0$, and hence
$$
\beta_R\geq \exp(-MR^2).
$$
Therefore,
$$
\limsup_{R\to\infty}\frac{\ln\beta_R}{R^2}\geq -M>-\infty.
$$
Thus, condition \eqref{eq:rv_gaussian_condition} follows from \eqref{con:ren}.

%and Corollary \ref{coro:regularVarying} yields the compact support property.

In contrast, for certain irregular functions $\beta$,  condition \eqref{eq:rv_gaussian_condition} is satisfied, but condition \eqref{con:ren} is not fulfilled.
The reason is that condition \eqref{eq:rv_gaussian_condition} is
subsequential in the radius variable, whereas \eqref{con:ren}  is pointwise and eventual in the spatial variable. We make this
distinction explicit by the following example.

Let $a_1>1$, define recursively
$$
a_{n+1}=a_n^4,
\qquad
b_n=a_n^2,
\qquad n\geq1.
$$
Then
$$
a_n<b_n<a_{n+1},
\qquad
\frac{b_n}{a_n}=a_n\to\infty.
$$
Choose a continuous non-decreasing function $h:[0,\infty)\to[0,\infty)$
such that $h(r)=b_1^2$ on $[0,b_1]$, such that
$h(r)=b_n^2$ on $[a_n,b_n]$ for every $n\geq2$, and such that $h$ is
linearly interpolated on each interval $[b_n,a_{n+1}]$. Define
$$
\beta(x):=\exp(-h(|x|)),
\qquad x\in\mathbb{R}^d.
$$
Then $\beta$ is positive, bounded, continuous, and radially non-increasing.
Consequently,
$$
\beta_R=\inf_{|x|\leq R}\beta(x)=
\exp(-h(R)).
$$
Taking $R=b_n$, we obtain
$$
\beta_{b_n}=\exp(-h(b_n))=
\exp(-b_n^2),$$
and hence
$$
\frac{\ln\beta_{b_n}}{b_n^2}=-1.
$$
Therefore,
$$
\limsup_{R\to\infty}\frac{\ln\beta_R}{R^2}\geq -1>-\infty.
$$
Thus the condition \eqref{eq:rv_gaussian_condition} is satisfied.

On the other hand, the classical pointwise Gaussian lower bound fails.
Indeed, at $r=a_n$ we have
$$
\beta(a_n e_1)=\exp(-h(a_n))=\exp(-b_n^2)=\exp(-a_n^4),
$$
where $e_1$ is any fixed unit vector. Hence
$$
\frac{\ln\beta(a_n e_1)}{a_n^2}=-a_n^2
\to -\infty.
$$
It follows that there are no constants $0<M<\infty$ and no $R_0>0$ such
that
$$ \beta(x)\geq \exp(-M|x|^2),
\qquad |x|\geq R_0.$$
Hence, condition \eqref{con:ren} fails. 

%Thus, even in the stable case
%$$
%\Psi(x,u)=\beta(x)u^{1+p},
%\qquad p\in(0,1],
%$$
%our criterion applies to spatially irregular branching rates which are not
%covered by the classical eventual pointwise Gaussian lower bound. In regular radial models, such as
%$\beta(x)\asymp\exp(-c|x|^\gamma)$, both conditions identify the same
%Gaussian order threshold, but the present formulation is genuinely more
%flexible for irregular profiles.

We next record two compact-support consequences of Theorem
\ref{th:csp_unify}. Moreover, since the lower bound branching mechanism $\Phi$ satisfies  Condition \ref{GreyCon} automatically, the global range of $X$ is also compact.  
\begin{coro}\label{coro2}
If the branching mechanism $\Psi$ of a super-Brownian motion $X$ satisfies Condition \ref{con_general}, then for any $\mu\in M_c$ and $t>0$,
			$$\mathbb{P}_\mu\left\{ \overline{\bigcup_{s \le t} S(X_s)} \text{ is compact} \right\} = 1.$$
 Moreover,
$$\mathbb{P}_{\mu}\{\mathscr{R}(X)\text{ is compact}\}=1.$$
\end{coro}
The proof of Corollary \ref{coro2} is given in Section \ref{sec4}.

\begin{remark}
If a branching mechanism $\Phi$ is spatially constant satisfying Condition \ref{con1}, then it satisfies Condition \ref{con_general}, and we recover the same compact support criterion of Sheu \cite{Sheu94,SHEU1997129}.
\end{remark}

\section{The speed of support propagation}\label{sec2}
\indent In this section, we prove Theorem \ref{th1}. The argument works directly under the local lower bound Condition \ref{con_localize}. The global lower bound and spatially constant cases are obtained by taking the local lower bound family to be independent of the radius.

The proof proceeds in four steps. First, we express the exit event through a weighted occupation time and its associated mild log-Laplace equation. Second, we dominate this solution by a local singular boundary blow-up solution. Third, a Feynman-Kac argument reduces the estimate to a Brownian exit probability of a given open ball. Finally, a finite covering of the initial support and the branching property yield the estimate for general compactly supported initial data.

\subsection{Preparations for the main results}
In this subsection, we prepare for the propagation speed estimate. Let $R>0$ and let $\psi_R$ be the radially symmetric function defined by
\begin{equation}\label{definpsi}
			\psi_R(x)=\left\{\begin{aligned}&0,\quad &| x|\leq R;\\
				&|x|/R-1,\quad &R<| x|\leq 2R;\\
				&1,\quad &| x|>2R.
			\end{aligned}\right.
\end{equation}
For a spatially dependent mechanism $\Psi$, let $u_\theta^\Psi(x,t)$ denote
the unique non-negative mild solution, locally bounded on each finite time
interval, of
\begin{equation}\label{PDE1Psi}
u_\theta^\Psi(x,t)=P_x\left[
\int_0^t\left\{\theta\psi_R(W_s)-\Psi(W_s,u_\theta^\Psi(W_s,t-s))
\right\}\dd s\right].
\end{equation}
The formal differential form of \eqref{PDE1Psi} is
\begin{equation}\label{PDE1Psi_formal}
\left\{
\begin{aligned}
\dfrac{\partial u_\theta^\Psi(x,t)}{\partial t}&=\dfrac{1}{2}\Delta u_\theta^\Psi(x,t)
-\Psi(x,u_\theta^\Psi(x,t))
+\theta \psi_R(x),
\quad (x,t)\in\mathbb{R}^d\times(0,\infty);\\
u_\theta^\Psi(x,0)&=0.
\end{aligned}
\right.
\end{equation}
When $\Psi(x,u)=\Phi(u)$ is spatially constant, we also write
$u_\theta^\Phi$ for the corresponding mild solution. It satisfies 
\begin{equation}\label{PDE1}
u_\theta^\Phi(x,t)=P_x\left[\int_0^t
\left\{\theta\psi_R(W_s)-\Phi(u_\theta^\Phi(W_s,t-s))
\right\}\dd s
\right].
\end{equation}
The formal differential form of \eqref{PDE1} is
\begin{equation}\label{PDE1_formal}
\left\{
\begin{aligned}
\dfrac{\partial u_\theta^\Phi(x,t)}{\partial t}
&=
\dfrac{1}{2}\Delta u_\theta^\Phi(x,t)
-\Phi(u_\theta^\Phi(x,t))
+\theta\psi_R(x),
\quad (x,t)\in\mathbb{R}^d\times(0,\infty);\\
u_\theta^\Phi(x,0)&=0.
\end{aligned}
\right.
\end{equation}
As mentioned above, the existence and uniqueness of the mild solutions to \eqref{PDE1Psi} and
\eqref{PDE1} are ensured by the theory of weighted occupation times; see \cite{Iscoe86} and \cite{Li22} for details.

The following proposition is the comparison principle used for the estimate of propagation probability.  It is proved in Appendix \ref{app_1} by a local parabolic barrier argument.
\begin{proposition}\label{prop_upperbound}
Suppose that $\Psi$ satisfies Condition \ref{con_localize}. Then, for $\theta>0$, $t\geq 0, R>0$ and $x\in B(0,R)$, we have
\begin{align}\label{stp2}
0\leq u_\theta^\Psi(x,t)\leq u_{R,\Phi_R}(x),
\end{align}
where $u_\theta^\Psi(x,t)$ is the unique non-negative mild solution of \eqref{PDE1Psi}, and $u_{R,\Phi_R}(x)$ is the unique radial classical solution of the singular boundary value problem  \eqref{SBVP} with $\Phi$ replaced by $\Phi_R$.
\end{proposition}

\begin{remark}
If Condition \ref{con_general} holds with lower bound $\Phi$, then one may take $\Phi_R=\Phi$ for all $R>0$. In that case, the mild solution $u_\theta^\Phi$ associated with the spatially constant lower bound $\Phi$ satisfies
$$
0\leq u_\theta^\Psi(x,t)\leq u_\theta^\Phi(x,t)\leq u_{R,\Phi}(x),
\qquad x\in B(0,R),\quad t\geq0,
$$
by the same comparison argument. The spatially constant case corresponds to $\Psi(x,u)=\Phi(u)$.
\end{remark}

Applying the Feynman-Kac formula and Proposition \ref{prop_upperbound}, we obtain the following pointwise estimate.
\begin{proposition}\label{prop_FK_estimate}
Suppose that $\Psi$ satisfies Condition \ref{con_localize}. Then for any $0<R_1<R$, we have
\begin{align}\label{FK_estimate}
u_\theta^\Psi(x,t)\leq u_{R,\Phi_R}(R_1)P_x(\sigma_{R_1}<t),\quad \forall x\in B(0,R_1),\ t>0,
\end{align}
where $P_x$ denotes the law of a $d$-dimensional Brownian motion $W$ starting from $x$ and $\sigma_{R_1}:=\inf\{s>0:|W_s|=R_1\}$.
\end{proposition}
\begin{proof}
Fix $x\in B(0,R_1)$ and $t>0$. Since $0\leq \psi_R\leq1$ and $\Psi(y,u)\geq0$ for $u\geq0$, the mild equation \eqref{PDE1Psi} gives the a priori bound
$$
0\leq u_\theta^\Psi(y,s)\leq \theta s\leq \theta t,
\qquad y\in\mathbb R^d,\quad 0\leq s\leq t.
$$
Define
$$
q_\theta(y,s):=
\begin{cases}
\dfrac{\Psi(y,u_\theta^\Psi(y,s))}{u_\theta^\Psi(y,s)},& u_\theta^\Psi(y,s)>0,\\[1.2ex]
0,& u_\theta^\Psi(y,s)=0.
\end{cases}
$$
Then $q_\theta\geq0$ and
$$
\Psi(y,u_\theta^\Psi(y,s))=q_\theta(y,s)u_\theta^\Psi(y,s).
$$
The value of $q_\theta$ on the set $\{u_\theta^\Psi=0\}$ is arbitrary for this product identity, and our definition is for convenience. By the L\'evy--Khintchine representation and the boundedness assumptions on the coefficients of $\Psi$, the local Lipschitz constant of $\Psi(y,\cdot)$ on $[0,\theta t]$ is bounded uniformly in $y\in\mathbb R^d$. Hence $q_\theta$ is a bounded non-negative measurable potential on $\mathbb R^d\times[0,t]$.

Using the identity
$\Psi(y,u_\theta^\Psi(y,s))=q_\theta(y,s)u_\theta^\Psi(y,s)$, the mild weighted
occupation time equation \eqref{PDE1Psi} may therefore be viewed  as the linear mild equation with source $\theta\psi_R$ and potential $q_\theta$ on this finite time interval. The Feynman-Kac representation for bounded measurable potentials gives
\begin{equation}\label{eq:global_FK_representation}
 u_\theta^\Psi(x,t)
 =P_x\left[\int_0^t \theta\psi_R(W_s)
 \exp\left\{-\int_0^s q_\theta(W_r,t-r)\dd r\right\}\dd s\right].
\end{equation}
Let
$$
\sigma_{R_1}:=\inf\{s>0:|W_s|=R_1\}.
$$
Since $R_1<R$, the function $\psi_R$ vanishes on $B(0,R_1)$. Conditioning on $\mathscr F_{\sigma_{R_1}\wedge t}$ in \eqref{eq:global_FK_representation}, we split the expectation according to $\{\sigma_{R_1}\geq t\}$ and $\{\sigma_{R_1}<t\}$. On $\{\sigma_{R_1}\geq t\}$, the Brownian path remains in $B(0,R_1)$ up to time $t$, and hence the contribution is zero. On $\{\sigma_{R_1}<t\}$, the same support property gives $\psi_R(W_s)=0$ for $0\leq s<\sigma_{R_1}$, and therefore
\begin{align}
 u_\theta^\Psi(x,t)
&=
P_x\left[
\mathbf{1}_{\{\sigma_{R_1}<t\}}
P_x\left[
\left.
\int_{\sigma_{R_1}}^t \theta\psi_R(W_s)
\exp\left\{-\int_0^s q_\theta(W_r,t-r)\dd r\right\}\dd s
\right|\mathscr F_{\sigma_{R_1}}
\right]
\right] \nonumber\\
&=
P_x\left[
\mathbf{1}_{\{\sigma_{R_1}<t\}}
\exp\left\{-\int_0^{\sigma_{R_1}}q_\theta(W_r,t-r)\dd r\right\}
I_{\sigma_{R_1}}
\right],
\label{eq:FK_split_conditional}
\end{align}
where
\begin{align*}
I_{\sigma_{R_1}}
:=P_x\left[
\left.
\int_{\sigma_{R_1}}^t \theta\psi_R(W_s)
\exp\left\{-\int_{\sigma_{R_1}}^s q_\theta(W_r,t-r)\dd r\right\}\dd s
\right|\mathscr F_{\sigma_{R_1}}
\right].
\end{align*}
By the strong Markov property of Brownian motion at $\sigma_{R_1}$, on $\{\sigma_{R_1}<t\}$ we have
\begin{align*}
I_{\sigma_{R_1}}
&=
P_{W_{\sigma_{R_1}}}\left[
\int_0^{t-\sigma_{R_1}}\theta\psi_R(W_s)
\exp\left\{-\int_0^s q_\theta(W_r,t-\sigma_{R_1}-r)\dd r\right\}\dd s
\right]\\
&=u_\theta^\Psi(W_{\sigma_{R_1}},t-\sigma_{R_1}).
\end{align*}
Substituting this identity into \eqref{eq:FK_split_conditional}, we obtain
\begin{equation}\label{Mar}
\begin{aligned}
u_\theta^\Psi(x,t)
=&P_x\left[ \mathbf{1}_{\{\sigma_{R_1}<t\}}
\exp\left\{-\int_0^{\sigma_{R_1}}q_\theta(W_s,t-s)\dd s\right\}
 u_\theta^\Psi(W_{\sigma_{R_1}},t-\sigma_{R_1})\right]\\
\leq& P_x\left[ \mathbf{1}_{\{\sigma_{R_1}<t\}}
 u_\theta^\Psi(W_{\sigma_{R_1}},t-\sigma_{R_1})\right].
\end{aligned}
\end{equation}
Here, the last inequality holds because $q_\theta\geq0$. By Proposition \ref{prop_upperbound}, on the event $\{\sigma_{R_1}<t\}$,
$$
u_\theta^\Psi(W_{\sigma_{R_1}},t-\sigma_{R_1})\leq u_{R,\Phi_R}(W_{\sigma_{R_1}})=u_{R,\Phi_R}(R_1),
$$
because $u_{R,\Phi_R}$ is radial. Combining this estimate with \eqref{Mar}, we deduce \eqref{FK_estimate}.
\end{proof}

\subsection{Propagation probability out of a ball}
In this subsection, we estimate the probability that the superprocess exits a ball before time $t$.

\begin{proposition}\label{prop1}
Suppose that $X$ is a super-Brownian motion with branching mechanism $\Psi$ satisfying Condition \ref{con_localize}. For any $\mu\in M_c$, there exists a positive constant $C(d)$ that depends only on the dimension $d$, such that for any $0\leq\rho(\mu)<R_1 < R$ and $t\in\left(0,\dfrac{(R_1-\rho(\mu))^2}{d}\right)$, we have
\begin{equation}\label{keyresult}
\begin{aligned}
&\mathbb{P}_{\mu}\{\exists s\leq t,X_s(\overline{B(0,R)}^c)>0\}\\
&\quad\leq \mu(1)\cdot u_{R,\Phi_R}(R_1)\cdot C(d) \left(\dfrac{R_1-\rho(\mu)}{\sqrt{t}}\right)^{d-2}\exp\left(-\dfrac{(R_1-\rho(\mu))^2}{2t}\right)\\
&\quad\leq \mu(1)\cdot G_{\Phi_R}^{-1}(R-R_1)\cdot C(d) \left(\dfrac{R_1-\rho(\mu)}{\sqrt{t}}\right)^{d-2}\exp\left(-\dfrac{(R_1-\rho(\mu))^2}{2t}\right).
\end{aligned}
\end{equation}
\end{proposition}
\begin{proof}
For simplicity, define
\begin{equation}
 p(t,\mu,R):=\mathbb{P}_{\mu}\{\exists s\leq t,X_s(\overline{B(0,R)}^c)>0\}.
\end{equation}
By the weighted occupation time representation
\eqref{eq:weighted_occupation_laplace}, applied with $f=\psi_R$, we have
$$
p(t,\mu,R)\leq \lim_{\theta\to\infty}\int_{\mathbb R^d}u_\theta^\Psi(x,t)\mu(\dd x),
$$
where $u_\theta^\Psi$ is the mild solution of \eqref{PDE1Psi}. Applying Proposition \ref{prop_FK_estimate} and using $S(\mu)\subset \overline{B(0,\rho(\mu))}$, we get
$$
p(t,\mu,R)
\leq
u_{R,\Phi_R}(R_1)\int_{S(\mu)}P_x(\sigma_{R_1}<t)\mu(\dd x).
$$
By the standard Brownian exit estimate, for $x\in S(\mu)$ and $t\in\left(0,(R_1-\rho(\mu))^2/d\right)$,
$$
P_x(\sigma_{R_1}<t)
\leq
C(d)\left(\frac{R_1-\rho(\mu)}{\sqrt{t}}\right)^{d-2}
\exp\left(-\frac{(R_1-\rho(\mu))^2}{2t}\right).
$$
Therefore,
$$
p(t,\mu,R)
\leq
\mu(1)u_{R,\Phi_R}(R_1)C(d)\left(\frac{R_1-\rho(\mu)}{\sqrt{t}}\right)^{d-2}
\exp\left(-\frac{(R_1-\rho(\mu))^2}{2t}\right).
$$
The second inequality in \eqref{keyresult} follows from the Keller estimate \eqref{eq:inverseUpperBound}.
\end{proof}

\subsection{Propagation probability out of a shifted ball}
To state the shifted estimate, let $\mu\in M_c$ and $x_0\in\mathbb R^d$. Define
\begin{equation}\label{def_rhox0}
\rho_{x_0}(\mu):=\inf\{r>0:S(\mu)\subset B(x_0,r)\}.
\end{equation}

\begin{coro}\label{corotrans}
Suppose that $X$ is a super-Brownian motion with branching mechanism $\Psi$ satisfying Condition \ref{con_localize}. Let $x_0\in\mathbb R^d$, $R>0$, and let $L>0$ be such that
$$
\overline{B(x_0,R)}\subset \overline{B(0,L)}.
$$
Then there exists a positive constant $C(d)$ which depends only on the dimension $d$, such that for any $0\leq\rho_{x_0}(\mu)<R_1 < R$ and $t\in\left(0,\dfrac{(R_1-\rho_{x_0}(\mu))^2}{d}\right)$, we have
\begin{equation}\label{shifted_estimate}
\begin{aligned}
&\mathbb{P}_{\mu}\{\exists s\leq t,X_s(\overline{B(x_0,R)}^c)>0\}\\
&\quad\leq \mu(1)\cdot u_{R,\Phi_L}(R_1)\cdot C(d) \left(\dfrac{R_1-\rho_{x_0}(\mu)}{\sqrt{t}}\right)^{d-2}\exp\left(-\dfrac{(R_1-\rho_{x_0}(\mu))^2}{2t}\right)\\
&\quad\leq \mu(1)\cdot G_{\Phi_L}^{-1}(R-R_1)\cdot C(d) \left(\dfrac{R_1-\rho_{x_0}(\mu)}{\sqrt{t}}\right)^{d-2}\exp\left(-\dfrac{(R_1-\rho_{x_0}(\mu))^2}{2t}\right).
\end{aligned}
\end{equation}
In particular, if $\mu=a\cdot\delta_{x_0}$ and $L$ satisfies $\overline{B(x_0,R)}\subset \overline{B(0,L)}$, then the same estimate holds with $\rho_{x_0}(\mu)=0$ and $\mu(1)=a$.
\end{coro}
\begin{proof}
Let
$$
\psi_{R,x_0}(x):=\psi_R(x-x_0).
$$
Let $u_{\theta,x_0}^\Psi$ be the mild solution of \eqref{PDE1Psi} with $\psi_R$ replaced by $\psi_{R,x_0}$. It is the log-Laplace cumulant associated with the weighted occupation time
$$
\theta\int_0^t\langle \psi_{R,x_0},X_s\rangle\dd s.
$$
Since $\overline{B(x_0,R)}\subset \overline{B(0,L)}$, Condition \ref{con_localize} gives
$$
\Psi(x,u)\geq \Phi_L(u),
\qquad x\in\overline{B(x_0,R)},\quad u\geq0.
$$
The shifted version of the local parabolic barrier in Lemma \ref{le:localBarrierPsi} gives
$$
0\leq u_{\theta,x_0}^\Psi(x,t)\leq u_{R,\Phi_L}(|x-x_0|),
\qquad x\in B(x_0,R),\quad t\geq0.
$$
Repeating the Feynman-Kac stopping argument in Proposition \ref{prop_FK_estimate}, with
$$
\sigma_{x_0,R_1}:=\inf\{s>0:|W_s-x_0|=R_1\},
$$
yields
$$
u_{\theta,x_0}^\Psi(x,t)
\leq
u_{R,\Phi_L}(R_1)P_x(\sigma_{x_0,R_1}<t),
\qquad x\in B(x_0,R_1),\quad t>0.
$$
The weighted occupation time representation gives
$$
\mathbb{P}_{\mu}\{\exists s\leq t,X_s(\overline{B(x_0,R)}^c)>0\}
\leq
\lim_{\theta\to\infty}\int_{\mathbb R^d}u_{\theta,x_0}^\Psi(x,t)\mu(\dd x).
$$
Using the preceding pointwise estimate and the Brownian exit estimate with distance $R_1-\rho_{x_0}(\mu)$ gives the first inequality in \eqref{shifted_estimate}. The Keller estimate gives the second inequality.
\end{proof}

\subsection{Proof of Theorem \ref{th1}}\label{subsec:proof_th1}						
Now we are ready to prove Theorem \ref{th1}.
\begin{proof}[Proof of Theorem \ref{th1}]
The key insight is that we can cover the initial support with balls of arbitrarily small radius $\delta$, and the final bound will be independent of this choice.
	
Fix $\delta > 0$ with $\delta\in(0, R_1-\sqrt{dt})$. This choice of $\delta$ ensures the application of Corollary \ref{corotrans}. Since $S(\mu)$ is compact, there exists a finite collection of points $\{x_1, \dots, x_N\} \subset S(\mu)$ such that
$$
S(\mu) \subset \bigcup_{i=1}^N B\left(x_i, \delta\right).
$$
Construct a disjoint decomposition $\{A_1, \dots, A_N\}$ of $S(\mu)$ where
$$
A_1 = S(\mu) \cap B\left(x_1, \delta\right), \quad A_i = S(\mu) \cap \left[B\left(x_i, \delta\right) \setminus \bigcup_{j=1}^{i-1} A_j\right] \text{ for } i = 2, \dots, N.
$$
Define restricted measures $\mu_i(A) = \mu(A \cap A_i)$. Then $\mu = \sum_{i=1}^N \mu_i$,  $S(\mu_i)\subset \overline{B(x_i,\delta)}$ and $\rho_{x_i}(\mu_i)\le \delta$.
							
By the branching property, $X$ can be decomposed into independent components $X_t = \sum_{i=1}^N X_t^{(i)}$. Then
\begin{equation}
\mathbb{P}_\mu\left\{\exists u \leq t: X_u((S(\mu)^R)^c) > 0\right\} \leq \sum_{i=1}^N \mathbb{P}_{\mu_i}\left\{\exists u \leq t: X_u^{(i)}((S(\mu)^R)^c) > 0\right\}.
\end{equation}
For each $i$, since $x_i\in S(\mu)$, we have
$$
\overline{B(x_i,R)}\subset S(\mu)^R,
$$
and hence
$$
(S(\mu)^R)^c\subset \overline{B(x_i,R)}^c.
$$
Moreover, $\overline{B(x_i,R)}\subset \overline{B(0,\rho(\mu)+R)}$. Applying Corollary \ref{corotrans} with center $x_i$, radius $R$, lower-bound index $L=\rho(\mu)+R$, and $R_1\in(\delta,R)$, we obtain
$$
\mathbb{P}_{\mu_i}\left\{\exists u \leq t:
X_u^{(i)}(\overline{B(x_i, R)}^c) > 0\right\}
\leq
\mu_i(1) u_{R,\Phi_{\rho(\mu)+R}}(R_1) C(d)
\left(\frac{R_1-\delta}{\sqrt t}\right)^{d-2}
\exp\left(-\frac{(R_1-\delta)^2}{2t}\right).
$$
Summing over $i$ gives the same bound with $\mu_i(1)$ replaced by $\mu(1)$. Letting $\delta\downarrow0$ within the interval $(0,R_1-\sqrt{dt})$ gives
$$
\mathbb{P}_\mu\left\{X_u\left((S(\mu)^R)^c\right) > 0 \text{ for some } u \leq t\right\}
\leq
\mu(1) u_{R,\Phi_{\rho(\mu)+R}}(R_1) C(d)
\left(\frac{R_1}{\sqrt t}\right)^{d-2}
\exp\left(-\frac{R_1^2}{2t}\right).
$$
The Keller estimate gives
$$
u_{R,\Phi_{\rho(\mu)+R}}(R_1)
\leq
G_{\Phi_{\rho(\mu)+R}}^{-1}(R-R_1),
$$
which proves the first assertion. The point-mass estimate is the special case of Corollary \ref{corotrans} with center $x$, initial measure $a\delta_x$, and any $L$ such that $\overline{B(x,R)}\subset\overline{B(0,L)}$; taking $L=|x|+R$ gives the stated bound.
\end{proof}	

\begin{remark}
Our upper bound on the propagation of support (Theorem \ref{th1}) generalizes \cite[Proposition 1.4 and Corollary 1.5]{dawson1994almost}, where  superprocesses with spatially constant stable branching mechanisms were considered and self-similarity applied in the proof. We instead utilize the weighted occupation time for superprocess and the mild comparison arguments for log-Laplace equations. Furthermore, there is a difference in the way estimates are derived for a process with general initial measure. The upper bound in \cite{dawson1994almost} is  established  by first considering a point mass at the origin \cite[Proposition 1.4]{dawson1994almost}, and subsequently extending the result to a general initial measure through weak convergence and the strong Feller property of $(2,d,\beta)$-superprocesses \cite[Corollary 1.5]{dawson1994almost}. In contrast, we find the upper bound directly for the superprocess with compactly supported initial measure by obtaining a uniform estimate in  Proposition \ref{prop_upperbound} via the Feynman-Kac formula and exploiting the branching property for the superprocess.
\end{remark}

\section{ Application: compact support property}\label{sec4}
As an application, we work directly with spatially dependent branching mechanisms satisfying the local lower bound Condition \ref{con_localize}. The global lower bound case in Condition \ref{con_general}, and hence the spatially constant case, will then follow as special cases.

Applying the localized speed estimate in Proposition \ref{prop1}, we can prove the compact-support criterion.
\begin{proof}[Proof of Theorem \ref{th:csp_unify}]
For $t, R>0$ define 
$$
A_R(t):=\left\{ \exists s \leq t : X_s \left(\overline{B(0,R)}^c\right) > 0  \right\}.
$$
Then for fixed $t>0$, $A_R(t)$ is decreasing with respect to $R$.
Let $R_1=R/2$ be large enough so that by Proposition \ref{prop1} we have
\begin{equation}\label{upper_bound}
\mathbb{P}_{\mu}(A_R(t))\leq \mu(1)\cdot G_{\Phi_R}^{-1}(R/2)\cdot C(d)\cdot \left( \frac{R/2-\rho(\mu)}{\sqrt{t}} \right)^{d-2} \exp \left( -\frac{(R/2-\rho(\mu))^2}{2t} \right).
\end{equation}
Under assumption \eqref{eq:sufficientCon}, there exist $M>0$ and a sequence $\{a_k\}$ increasing to $+\infty$ such that for all large $k$
$$
\dfrac{\ln G_{\Phi_{a_k}}^{-1}(a_k/2)}{a_k^2}<M,
$$
which is equivalent to 
$$
G_{\Phi_{a_k}}^{-1}(a_k/2)<\exp(M a_k^2).
$$
Choosing $t_0>0$ satisfying $M<\frac{1}{32 t_0}$, we have
$$
G_{\Phi_{a_k}}^{-1}(a_k/2)\left( \frac{a_k/2-\rho(\mu)}{\sqrt{t_0}} \right)^{d-2} \exp \left( -\frac{(a_k/2-\rho(\mu))^2}{2t_0} \right)\to0\quad \text{as }k\to\infty.
$$
It then follows from \eqref{upper_bound} that 
$$
\lim_{k\to\infty}\mathbb{P}_{\mu}(A_{a_k}(t_0))= 0.
$$
Since $A_R(t)$ decreases in $R$, we have for any $0\leq t< t_0$
$$
\mathbb{P}_{\mu} \left( \lim_{R\to\infty} A_R(t) \right) = \lim_{k \to \infty} \mathbb{P}_{\mu}(A_{a_k}(t)) = 0.
$$
Hence
$$
\mathbb{P}_\mu\left\{ \overline{\bigcup_{s \le t} S(X_s)} \text{ is compact} \right\} = 1.
$$

For any fixed target time $t > 0$, we partition the time interval $[0, t]$ into $N$ sub-intervals of equal length $\Delta t = t/N < t_0$. Let $t_k = k \Delta t$ for $k = 0, 1, \dots, N$.
For each $k$, define
$$
\rho_k := \sup_{s \le t_k} \inf \{ r > 0 : S(X_s) \subseteq B(0, r) \}.
$$
We shall prove by induction on $k$ that, for all $k = 1, \dots, N$,
\begin{equation}\label{induction}
  \mathbb{P}_\mu(\rho_k < \infty) = 1 ,\quad \forall \mu\in M_c.
\end{equation} 
For $k=1$, the previous discussion gives $\mathbb{P}_\mu(\rho_1 < \infty) = 1$ and $X_{t_1}\in M_c$ almost surely. Now assume that \eqref{induction} holds for some $k\geq 1$. Then $X_{t_k}\in M_c$ almost surely. Applying the strong Markov property of the super-Brownian motion, we have
$$
\begin{aligned}
    &\mathbb{P}_\mu \left( \overline{\bigcup_{s \in [0,\Delta t]} S(X_{t_k+s})} \text{ is compact}\;\middle|\; \mathcal{F}_{t_k} \right)\\
    &= \mathbb{P}_{X_{t_k}} \left( \overline{\bigcup_{s \in [0,\Delta t]} S(\tilde{X}_s)} \text{ is compact} \right)=1,
\end{aligned}
$$
where $\tilde{X}$ is a super-Brownian motion with the same branching mechanism as $X$. Thus
$$
\mathbb{P}_\mu(\rho_{k+1} < \infty) = 1.
$$
The desired result follows. 
\end{proof}

Now, we turn to the case that the branching mechanism $\Psi$ satisfies Condition \ref{con_general}.
\begin{proof}[Proof of Corollary \ref{coro2}]
Since $\Psi$ satisfies Condition \ref{con_general}, it satisfies Condition \ref{con_localize} with $\Phi_R=\Phi$ for every $R>0$. Moreover, $G_{\Phi}^{-1}(R/2)\to0$ as $R\to\infty$, and hence
$$
\liminf_{R\to\infty}\dfrac{\ln G_{\Phi}^{-1}(R/2)}{R^2}<\infty.
$$
Theorem \ref{th:csp_unify} therefore gives the compact support property.

To establish the compactness of the range $\mathscr{R}(X)$, it remains to show that $X$ becomes extinct in finite time almost surely.
This follows directly from the cumulant comparison. For $\lambda>0$, let
$u_\lambda(x,t)$ be the log-Laplace cumulant corresponding to the test
function $f\equiv\lambda$, so that
$$
\mathbb E_\mu\left[\exp\{-\lambda X_t(1)\}\right]
=
\exp\left\{-\langle u_\lambda(\cdot,t),\mu\rangle\right\}.
$$
Then $u_\lambda$ solves
$$
\partial_t u_\lambda
=
\frac12\Delta u_\lambda-\Psi(x,u_\lambda),
\qquad
u_\lambda(x,0)=\lambda .
$$
Let $v_\lambda(t)$ be the cumulant of the $\Phi$-CSBP, namely
$$
\frac{\dd}{\dd t}v_\lambda(t)=-\Phi(v_\lambda(t)),
\qquad
v_\lambda(0)=\lambda .
$$
Since $\Psi(x,u)\ge\Phi(u)$, the spatially constant function
$v_\lambda(t)$ is a supersolution of the equation for $u_\lambda$. Hence,
by comparison,
$$
u_\lambda(x,t)\le v_\lambda(t),
\qquad x\in\mathbb R^d,
\quad t\ge0 .
$$
Consequently,
$$
\mathbb E_\mu\left[\exp\{-\lambda X_t(1)\}\right]
\ge
\exp\{-\mu(1)v_\lambda(t)\}.
$$
Letting $\lambda\to\infty$, we obtain
$$
\mathbb P_\mu(X_t(1)=0)
\ge
\exp\{-\mu(1)v_\infty(t)\},
\qquad
v_\infty(t):=\lim_{\lambda\to\infty}v_\lambda(t).
$$
By Grey's condition together with the critical or subcritical assumption,
$v_\infty(t)<\infty$ for every $t>0$, and
$$
v_\infty(t)\downarrow0,
\qquad t\to\infty .
$$
Therefore
$$
\mathbb P_\mu(\zeta<\infty)
=
\lim_{t\to\infty}\mathbb P_\mu(X_t(1)=0)
\ge
\lim_{t\to\infty}\exp\{-\mu(1)v_\infty(t)\}
=1,
$$
where
$$
\zeta:=\inf\{t>0:X_t(1)=0\}.
$$
Thus $X$ becomes extinct in finite time almost surely.
    
Notice that $X_t$ is the null measure for all $t \ge \zeta$. Therefore, the total range collapses to the range over a finite time interval and
$$
\mathscr{R}(X) = \overline{\bigcup_{0 \le t \le \zeta} S(X_t)}.
$$
For any integer $k \in \mathbb{N}$, on the event $\{\zeta \le k\}$, we have $\mathscr{R}(X) \subset \overline{\bigcup_{0 \le t \le k} S(X_t)}$. From the compact-support property on $[0,k]$,
$$
\mathbb P_\mu\left(
\bigcup_{n=1}^\infty
\left\{\forall s\le k,\ S(X_s)\subset \overline{B(0,n)}\right\}
\right)=1.
$$ 
Consequently,
$$
\mathbb{P}_\mu(\mathscr{R}(X) \text{ is compact}) \ge \mathbb{P}_\mu(\{ \mathscr{R}(X) \text{ is compact} \} \cap \{\zeta \le k\}) = \mathbb{P}_\mu(\zeta \le k).
$$
Taking the limit as an integer $k \to \infty$, we obtain
$$
\mathbb{P}_\mu(\mathscr{R}(X) \text{ is compact}) \ge \lim_{k \to \infty} \mathbb{P}_\mu(\zeta \le k) = \mathbb{P}_\mu(\zeta < \infty) = 1.
$$
\end{proof}

\begin{coro}
 If the branching mechanism $\Psi$ of a super-Brownian motion $X$ satisfies Condition \ref{con_general}, then for any $\mu\in M_F$,
$$
\mathbb{P}_{\mu}\{S(X_t)\text{ is compact for all } t>0\}=1.
$$   
\end{coro}
                
\begin{proof}
For $R>0$, let $\mu_1(\cdot)=\mu(\cdot\cap B(0,R))$ and $\mu_2(\cdot)=\mu(\cdot\cap B(0,R)^c)$. Then $\mu=\mu_1+\mu_2$ and $\mu_1\in M_c$. Then the super-Brownian motion $X$ started from $\mu$ can be decomposed into 
$$
X_t=X^{(1)}_t+X_t^{(2)}
$$
where $X^{(i)}, i=1,2$ are independent super-Brownian motions with branching mechanisms $\Psi$ starting from $\mu_i, i=1,2$. 
By the branching property, for $t>0$, 
\begin{equation}
\mathbb{P}_\mu\{S(X_t) \text{ is compact}\}= \mathbb{P}_{\mu_1}\{S(X_t^{(1)}) \text{ is compact}\}\mathbb{P}_{\mu_2}\{S(X_t^{(2)}) \text{ is compact}\}.
\end{equation}
By Corollary \ref{coro2}, 
$$
\mathbb{P}_{\mu_1}\{S(X_t^{(1)}) \text{ is compact}\}=1.
$$
Notice that 
$$
\{S(X_t^{(2)}) \text{ is compact}\}\supset \{X_t^{(2)}=0\},
$$
we have
$$
\mathbb{P}_\mu\{S(X_t) \text{ is compact}\}\geq\mathbb{P}_{\mu_2}\{X_t^{(2)}=0\}.
$$
Now we estimate the probability of extinction.
                                
Recall that the extinction probability of $X_t^{(2)}$ is given by
\begin{equation}
\mathbb{P}_{\mu_2}\{X_t^{(2)}=0\}=\exp\left\{-\int_{\mathbb{R}^d}v(x,t)\mu_2(dx)\right\}
\end{equation}
where $v(x,t)$ is the positive solution of \eqref{deq} with $f(x)\equiv\infty$. If we denote the cumulant semigroup of $\Phi$-CSBP by $v_t(\lambda)$ and set $v_t=\lim_{\lambda\to\infty}v_t(\lambda)$, then, applying the comparison principle, we see that $v(x,t)$ is bounded above by $v(t)$. Since $\Phi$ satisfies Grey's condition, we have $v(t)<\infty$ for all $t$. Hence, we have
\begin{equation}
\mathbb{P}_\mu\{S(X_t) \text{ is compact}\}\geq \mathbb{P}_{\mu_2}\{X_t^{(2)}=0\}\geq \exp(-\left<\mu_2,1\right> v(t)).
\end{equation}
Since the above inequality holds for any decomposition of $\mu$, letting $R\to\infty$, we have $\left<\mu_2,1\right>\to0$, which implies that
$$
\mathbb{P}_\mu\{S(X_t) \text{ is compact}\}\geq 1.
$$
This proves the following. 
$$
\mathbb{P}_\mu\{S(X_t) \text{ is compact}\}= 1,
\quad \forall t>0.
$$
For each integer $n\geq1$, apply the preceding argument with
$t=1/n$. Since
$$
\mathbb{P}_\mu\{S(X_{1/n})\text{ is compact}\}=1,
$$
Applying the Markov property and the finite-time compact support property to compactly
supported initial measures repeatedly gives
$$
\mathbb{P}_{\mu}\{S(X_t)\text{ is compact for all }t\geq 1/n\}=1.
$$
Taking the countable intersection over $n\geq1$, we obtain
$$
\mathbb{P}_{\mu}\{S(X_t)\text{ is compact for all }t>0\}=1.
$$
\end{proof}

\subsection{Proof of Corollary \ref{coro:regularVarying}}\label{sec:proof_regular_varying}

\begin{proof}[Proof of Corollary \ref{coro:regularVarying}]
Set
$$
\Phi_R(u):=\beta_R\Phi(u),\qquad u\geq0.
$$
Since $\beta_R>0$ and $\Phi$ satisfies Condition \ref{con1}, each
$\Phi_R$ also satisfies Condition \ref{con1}. Moreover, for every $R>0$,
$$
\Psi(x,u)\geq \Phi_R(u),
\qquad x\in\overline{B(0,R)},\quad u\geq0.
$$
Thus Condition \ref{con_localize} holds. It remains to verify the inverse
Keller integral condition \eqref{eq:sufficientCon}.

Let
$$
F(z):=\int_0^z \Phi(v)\dd v.
$$
Since $\Phi$ is regularly varying at infinity with index $1+p$, where
$p\in(0,1]$, Karamata's theorem \cite[Chapter 1]{BGT87} gives
$$
F(z)\sim \frac{z \Phi(z)}{2+p},
\qquad z\to\infty.
$$
In particular, $F$ is regularly varying at infinity with index $2+p$.
Consequently,
$$
\frac{1}{\sqrt{2F(z)}}
$$
is regularly varying at infinity with index $-(2+p)/2=-1-p/2$. By
Karamata's theorem for tails,
$$
G_\Phi(z):=
\int_z^\infty \frac{\dd s}{\sqrt{2F(s)}}
$$
is regularly varying at infinity with index $-p/2$. Hence $G_{\Phi}^{-1}$ is
regularly varying at $0$ with index $-2/p$.

Since
$$
G_{\Phi_R}(z)
=
\int_z^\infty
\frac{\dd s}{\sqrt{2\int_0^s\beta_R \Phi(v)\dd v}}
=
\beta_R^{-1/2}G_\Phi(z),
$$
we have
\begin{equation}\label{eq:rv_inverse_scaling_detail}
G_{\Phi_R}^{-1}(R/2)
=
G_{\Phi}^{-1}\left(\frac{R\sqrt{\beta_R}}{2}\right).
\end{equation}

Assume that
$$
\limsup_{R\to\infty}\frac{\ln \beta_R}{R^2}>-\infty.
$$
Then there exist a sequence $R_n\to\infty$ and a constant $C<\infty$ such
that
\begin{align} \label{eq:exponentialsubsequence}
\beta_{R_n}\geq \exp(-C R_n^2).
\end{align}

If $R_n\sqrt{\beta_{R_n}}$ is bounded away from $0$ along a subsequence,
then there exists $\delta>0$ such that
$$
\frac{R_n\sqrt{\beta_{R_n}}}{2}\geq \delta
$$
along this subsequence. Since $G_{\Phi}^{-1}$ is decreasing, we have
$$
G_{\Phi}^{-1}\left(\frac{R_n\sqrt{\beta_{R_n}}}{2}\right)
\leq G_{\Phi}^{-1}(\delta)<\infty.
$$
Hence \eqref{eq:rv_inverse_scaling_detail} gives
$$
\liminf_{R\to\infty}
\frac{\ln G_{\Phi_R}^{-1}(R/2)}{R^2}<\infty.
$$

Hence it remains to consider
the case in which, after passing to a subsequence if necessary,
$$
R_n\sqrt{\beta_{R_n}}\to0.
$$

Since $G_{\Phi}^{-1}$ is regularly varying at $0$ with index $-2/p$, Potter's
bound \cite[Theorem 1.5.6]{BGT87} implies that, for every $\varepsilon>0$,
there exists $C_\varepsilon<\infty$ such that, for all sufficiently large
$n$,
$$
G_{\Phi}^{-1}\left(\frac{R_n\sqrt{\beta_{R_n}}}{2}\right)
\leq
C_\varepsilon
\left(\frac{R_n\sqrt{\beta_{R_n}}}{2}\right)^{-2/p-\varepsilon}.
$$
Combining \eqref{eq:rv_inverse_scaling_detail} and \eqref{eq:exponentialsubsequence}, we obtain
$$
\begin{aligned}
\ln G_{\Phi_{R_n}}^{-1}(R_n/2)
&\leq
\ln C_\varepsilon
+\left(\frac{2}{p}+\varepsilon\right)
\left|
\ln\left(\frac{R_n\sqrt{\beta_{R_n}}}{2}\right)
\right|  \\
&\leq
\left(\frac{1}{p}+\frac{\varepsilon}{2}\right)C R_n^2
+O(\ln R_n).
\end{aligned}
$$
Therefore,
$$
\liminf_{R\to\infty}
\frac{\ln G_{\Phi_R}^{-1}(R/2)}{R^2}<\infty.
$$
Thus \eqref{eq:sufficientCon} holds. The assertion follows from Theorem
\ref{th:csp_unify}.

\end{proof}

\appendix
\section{Appendix}
\subsection{Local parabolic barriers}\label{app_1}

In this appendix, we collect the local comparison estimate used in the proof
of Proposition \ref{prop_upperbound}. The argument is local in a ball. It
uses the boundary blow-up of the elliptic Keller solution and a comparison
argument for the stopped mild equations. In particular, no classical
regularity of the log-Laplace solution is needed.

\begin{lemma}[Local parabolic barrier on a ball]\label{le:localBarrierPsi}
Let $\Phi_0$ be  a spatially constant branching mechanism  satisfying  Condition \ref{con1} and $\Psi$ be another  branching mechanism  such that
$$
\Psi(x,u)\geq \Phi_0(u),
\qquad x\in\overline{B(x_0,R)},\quad u\geq0
$$
for $x_0\in\mathbb{R}^d$ and $R>0$.
Set $\psi_{R,x_0}(x):=\psi_R(x-x_0)$ for $\psi_R$  defined in \eqref{definpsi}. Given $\theta>0$, let $u_\theta(x,t)$ be the
unique non-negative mild solution, locally bounded on each finite time interval, of
\begin{equation}\label{eq:PDE_shifted_appendix}
u_\theta(x,t)
=
P_x\left[
\int_0^t
\left\{
\theta\psi_{R,x_0}(W_s)
-
\Psi(W_s,u_\theta(W_s,t-s))
\right\}\dd s
\right]
\end{equation}
and let $u_{R,\Phi_0}$ denote the radial classical solution of \eqref{SBVP} in $B(0,R)$ associated with $\Phi_0$. Then, we have
\begin{equation}\label{eq:localBarrierPsi_alltime}
u_\theta(x,t)\leq u_{R,\Phi_0}(|x-x_0|),
\qquad x\in B(x_0,R),\quad t\geq0 .
\end{equation}
\end{lemma}

\begin{proof}
First, since $\psi_{R,x_0}\leq1$ and $\Psi(x,u)\geq0$ for $u\geq0$, the
mild equation \eqref{eq:PDE_shifted_appendix} gives
\begin{equation}\label{eq:uthetaPsi_uniform_bound}
0\leq u_\theta(x,t)\leq \theta t,
\qquad x\in\mathbb{R}^d,\quad t\geq0 .
\end{equation}
In particular,
\begin{equation}\label{eq:uthetaPsi_uniform_bound_T}
0\leq u_\theta(x,t)\leq \theta T,
\qquad x\in\mathbb{R}^d,\quad 0\leq t\leq T .
\end{equation}

Fix $0<R_0<R$ and $T>0$. Define
$$
U(x):=u_{R,\Phi_0}(|x-x_0|),
\qquad x\in B(x_0,R).
$$
By the elliptic Keller construction, $U\in C^2(B(x_0,R))$,
$$
\frac{1}{2}\Delta U=\Phi_0(U)
\qquad \text{in }B(x_0,R),
$$
and $U(x)\to+\infty$ as $|x-x_0|\uparrow R$. Hence we may choose $r$ with
$R_0<r<R$ such that
\begin{equation}\label{eq:choose_r_Psi}
u_{R,\Phi_0}(r)>\theta T .
\end{equation}
Set
$$
D:=B(x_0,r),
\qquad
\tau_D:=\inf\{s>0:W_s\notin D\}.
$$
Since Brownian paths are continuous, $W_{\tau_D}\in\partial D$ on
$\{\tau_D<\infty\}$. By \eqref{eq:uthetaPsi_uniform_bound_T} and
\eqref{eq:choose_r_Psi}, we have
\begin{equation}\label{eq:lateral_boundary_Psi}
u_\theta(x,t)\leq \theta T<u_{R,\Phi_0}(r)=U(x),
\qquad x\in\partial D,\quad 0\leq t\leq T.
\end{equation}
Also,
\begin{equation}\label{eq:initial_boundary_Psi}
u_\theta(x,0)=0\leq U(x),
\qquad x\in D.
\end{equation}

We now compare $u_\theta$ and $U$ on $D\times[0,T]$ in the mild sense.
Since $r<R$, $\psi_{R,x_0}=0$ on $D$. Let
$$
M:=\theta T+\sup_{x\in\overline D}U(x).
$$
By the L\'evy--Khintchine representation and the boundedness assumptions on
the coefficients of $\Psi$, the function $\Psi(x,\cdot)$ is Lipschitz on
$[0,M]$, uniformly for $x\in\overline D$. Indeed, for $0\leq z\leq M$,
$$
\left|\partial_u\Psi(x,z)\right|
\leq
\|\alpha\|_\infty
+
2M\|\beta\|_\infty
+
\sup_{y\in\mathbb{R}^d}
\int_0^\infty \lambda(1-e^{-\lambda M})\pi(y,\dd\lambda)<\infty .
$$
Choose $L>0$ such that
\begin{equation}\label{eq:Lipschitz_Psi_appendix}
|\Psi(x,a)-\Psi(x,b)|\leq L|a-b|,
\qquad x\in\overline D,\quad 0\leq a,b\leq M.
\end{equation}
Define
$$
H(x,z):=Lz-\Psi(x,z),
\qquad x\in\overline D,\quad 0\leq z\leq M.
$$
Then $H(x,\cdot)$ is non-decreasing on $[0,M]$, uniformly in
$x\in\overline D$. Indeed, if $0\leq a\leq b\leq M$, then
$$
H(x,b)-H(x,a)
=
L(b-a)-\{\Psi(x,b)-\Psi(x,a)\}\geq0
$$
by \eqref{eq:Lipschitz_Psi_appendix}. Moreover, $H$ is Lipschitz in the
second variable on this interval; in particular,
\begin{equation}\label{eq:H_Lipschitz_appendix}
|H(x,a)-H(x,b)|\leq 2L|a-b|,
\qquad x\in\overline D,\quad 0\leq a,b\leq M.
\end{equation}

We use the stopped mild formulation on the ball $D$. We first spell out the
dynamic programming step. Put
$$
F(y,s):=\theta\psi_{R,x_0}(y)-\Psi(y,u_\theta(y,s)),
\qquad y\in\mathbb{R}^d,\quad 0\leq s\leq T .
$$
Then \eqref{eq:PDE_shifted_appendix} can be written as
$$
u_\theta(x,t)=P_x\left[\int_0^t F(W_s,t-s)\dd s\right].
$$
Let
$$
\tau:=t\wedge\tau_D .
$$
Splitting the integral at $\tau$, we obtain
\begin{align}
u_\theta(x,t)
&=
P_x\left[
\int_0^\tau F(W_s,t-s)\dd s
\right]+P_x\left[
\mathbf 1_{\{\tau_D<t\}}
\int_{\tau_D}^t F(W_s,t-s)\dd s
\right].
\label{eq:dynamic_split_appendix}
\end{align}
By the strong Markov property of Brownian motion at $\tau_D$, the second
term in \eqref{eq:dynamic_split_appendix} is
\begin{align}
 P_x\left[
\mathbf 1_{\{\tau_D<t\}}
\int_{\tau_D}^t F(W_s,t-s)\dd s
\right] 
&=P_x\left[\mathbf 1_{\{\tau_D<t\}}
P_{W_{\tau_D}}\left[
\int_0^{t-\tau_D}F(W_s,t-\tau_D-s)\dd s
\right]
\right] \nonumber\\
&=P_x\left[
\mathbf 1_{\{\tau_D<t\}}
u_\theta(W_{\tau_D},t-\tau_D)
\right].\label{eq:dynamic_after_exit_appendix}
\end{align}
On the other hand, since $W_s\in D$ for $0\leq s<\tau$ and
$\psi_{R,x_0}=0$ in $D$, we have
$$
F(W_s,t-s)=-\Psi(W_s,u_\theta(W_s,t-s)).
$$
By the definition
$$H(x,z):=Lz-\Psi(x,z),$$
this becomes
$$
F(W_s,t-s)=-Lu_\theta(W_s,t-s)+H(W_s,u_\theta(W_s,t-s)),
\qquad 0\leq s<\tau .
$$
Combining this identity with \eqref{eq:dynamic_split_appendix} and
\eqref{eq:dynamic_after_exit_appendix}, and using the convention
$u_\theta(W_t,0)=0$ on $\{\tau_D\geq t\}$, gives the undiscounted stopped
mild identity
\begin{align}\label{eq:undiscounted_stopped_appendix}
u_\theta(x,t)
=P_x\left[u_\theta(W_\tau,t-\tau)\right] +
P_x\left[\int_0^\tau\left\{-Lu_\theta(W_s,t-s)+H(W_s,u_\theta(W_s,t-s))
\right\}\dd s
\right].
\end{align}
Equivalently, applying the variation-of-constants formula to
\eqref{eq:undiscounted_stopped_appendix} yields
\begin{align}\label{eq:stopped_mild_u_appendix}
u_\theta(x,t)=P_x\left[e^{-L(t\wedge\tau_D)}u_\theta(W_{t\wedge\tau_D},t-t\wedge\tau_D)
\right]+P_x\left[\int_0^{t\wedge\tau_D}e^{-Ls}H(W_s,u_\theta(W_s,t-s))\dd s
\right].
\end{align}
Indeed, the last identity follows by applying
\eqref{eq:undiscounted_stopped_appendix} on a partition of
$[0,t\wedge\tau_D]$, multiplying the contribution starting at time $s$ by
$e^{-Ls}$, summing the resulting telescoping expression, and then letting
the mesh tend to zero. The passage to the limit is justified by
\eqref{eq:uthetaPsi_uniform_bound_T} and the boundedness of $H$ on
$\overline D\times[0,M]$.

On the other hand, since $U$ is time-independent and
$\frac12\Delta U=\Phi_0(U)$ in $D$, we apply the exponentially weighted
Dynkin formula to the stopped Brownian motion. This is just It\^o's formula
applied to $e^{-Ls}U(W_s)$ up to the bounded stopping time
$$
\sigma:=t\wedge\tau_D .
$$
Equivalently, one may use the standard Dynkin formula with killing rate $L$;
see, for example, \cite[Theorem 7.3.3 and Theorem 7.4.1]{Oksendal03}.

Since $r<R$, the function $U$ is $C^2$ in a neighbourhood of $\overline D$.
Thus $U$, $\nabla U$, and $\Delta U$ are bounded on $\overline D$. It\^o's
formula gives
\begin{align}\label{eq:ito_U_appendix}
e^{-L\sigma}U(W_\sigma)=U(x)+\int_0^\sigma e^{-Ls}\nabla U(W_s)\cdot \dd W_s+
\int_0^\sigma e^{-Ls}\left\{\frac12\Delta U(W_s)-LU(W_s)
\right\}\dd s .
\end{align}
The stochastic integral in \eqref{eq:ito_U_appendix} is a true martingale with mean zero, because the integrand is bounded on $[0,\sigma]$. Taking
$P_x$-expectations and rearranging, we obtain
\begin{align}\label{eq:dynkin_U_appendix}
U(x)=P_x\left[e^{-L(t\wedge\tau_D)}
U(W_{t\wedge\tau_D})
\right]+P_x\left[
\int_0^{t\wedge\tau_D}
e^{-Ls}
\left\{
LU(W_s)-\frac12\Delta U(W_s)
\right\}\dd s
\right].
\end{align}
Since $\frac12\Delta U=\Phi_0(U)$ in $D$, and the value at the single time
$\tau_D$ is irrelevant for the time integral, \eqref{eq:dynkin_U_appendix}
becomes
\begin{align}\label{eq:stopped_mild_U_exact_appendix}
U(x)=P_x\left[e^{-L(t\wedge\tau_D)}U(W_{t\wedge\tau_D})\right]+
P_x\left[\int_0^{t\wedge\tau_D}e^{-Ls}\{LU(W_s)-\Phi_0(U(W_s))\}\dd s
\right].
\end{align}

Since $\overline D\subset\overline{B(x_0,R)}$ and
$$
\Psi(y,u)\geq\Phi_0(u),
\qquad y\in\overline{B(x_0,R)},\quad u\geq0,
$$
we have
$$
LU(y)-\Phi_0(U(y))
\geq
LU(y)-\Psi(y,U(y))
=
H(y,U(y)),
\qquad y\in D.
$$
Therefore \eqref{eq:stopped_mild_U_exact_appendix} yields
\begin{align}\label{eq:stopped_mild_U_appendix}
U(x)\geq P_x\left[e^{-L(t\wedge\tau_D)}
U(W_{t\wedge\tau_D})
\right] +P_x\left[
\int_0^{t\wedge\tau_D}
e^{-Ls}H(W_s,U(W_s))\dd s
\right].
\end{align}

Set
$$
w(x,t):=u_\theta(x,t)-U(x),
\qquad x\in D,\quad 0\leq t\leq T.
$$
Subtracting \eqref{eq:stopped_mild_U_appendix} from
\eqref{eq:stopped_mild_u_appendix}, and using
\eqref{eq:lateral_boundary_Psi} and \eqref{eq:initial_boundary_Psi} for the
terminal terms, we obtain
\begin{align}
w(x,t)
&\leq
P_x\left[
\int_0^{t\wedge\tau_D}
e^{-Ls}
\{H(W_s,u_\theta(W_s,t-s))-H(W_s,U(W_s))\}\dd s
\right].
\label{eq:w_mild_appendix}
\end{align}
The terminal contribution is non-positive: on $\{\tau_D<t\}$ this follows
from \eqref{eq:lateral_boundary_Psi}, while on $\{\tau_D\geq t\}$ it equals
$-e^{-Lt}U(W_t)\leq0$.

By the monotonicity of $H(x,\cdot)$ and \eqref{eq:H_Lipschitz_appendix},
for all $x\in\overline D$ and $0\leq a,b\leq M$,
$$
H(x,a)-H(x,b)\leq 2L(a-b)^+.
$$
Hence \eqref{eq:w_mild_appendix} gives
$$
w^+(x,t)
\leq
2L\,
P_x\left[
\int_0^{t\wedge\tau_D}
w^+(W_s,t-s)\dd s
\right],
\qquad x\in D,\quad 0\leq t\leq T.
$$
Let
$$
m(t):=\sup_{0\leq s\leq t}\sup_{x\in D}w^+(x,s),
\qquad 0\leq t\leq T.
$$
Then
$$
m(t)\leq 2L\int_0^t m(t-s)\dd s
=
2L\int_0^t m(s)\dd s.
$$
By Gronwall's lemma, $m(t)=0$ for $0\leq t\leq T$. Consequently,
$$
u_\theta(x,t)\leq U(x)=u_{R,\Phi_0}(|x-x_0|),
\qquad x\in D,\quad 0\leq t\leq T.
$$
Since $\overline{B(x_0,R_0)}\subset D$, we deduce that for every $0<R_0<R$ and every $T>0$,
\begin{equation}
u_\theta(x,t)\leq u_{R,\Phi_0}(|x-x_0|),
\qquad x\in \overline{B(x_0,R_0)},\quad 0\leq t\leq T .
\end{equation}
Finally, since $R_0<R$ and $T>0$ were arbitrary, we obtain
\eqref{eq:localBarrierPsi_alltime}.
\end{proof}

\begin{proof}[Proof of Proposition \ref{prop_upperbound}]
Apply Lemma \ref{le:localBarrierPsi} with $x_0=0$ and $\Phi_0=\Phi_R$. The
hypothesis of the lemma is exactly Condition \ref{con_localize} on
$\overline{B(0,R)}$. Since $\psi_{R,0}=\psi_R$, the solution $u_\theta$ in
the lemma is $u_\theta^\Psi$. Hence
$$
0\leq u_\theta^\Psi(x,t)\leq u_{R,\Phi_R}(x),
\qquad x\in B(0,R),\quad t\geq0,
$$
which is \eqref{stp2}.
\end{proof}

\bibliographystyle{abbrv}
\bibliography{SBMbib}

@article {Delmas99,
    AUTHOR = {Delmas, Jean-Fran\c{c}ois},
     TITLE = {Path properties of superprocesses with a general branching
              mechanism},
   JOURNAL = {Ann. Probab.},
  FJOURNAL = {The Annals of Probability},
    VOLUME = {27},
      YEAR = {1999},
    NUMBER = {3},
     PAGES = {1099--1134},
      ISSN = {0091-1798,2168-894X},
   MRCLASS = {60G57 (60J25 60J55 60J80)},
  MRNUMBER = {1733142},
MRREVIEWER = {Klaus\ Fleischmann},
       DOI = {10.1214/aop/1022677441},
       URL = {https://doi.org/10.1214/aop/1022677441},
}

@article {Keller57,
    AUTHOR = {Keller, J. B.},
     TITLE = {On solutions of {$\Delta u=f(u)$}},
   JOURNAL = {Comm. Pure Appl. Math.},
  FJOURNAL = {Communications on Pure and Applied Mathematics},
    VOLUME = {10},
      YEAR = {1957},
     PAGES = {503--510},
      ISSN = {0010-3640,1097-0312},
   MRCLASS = {35.0X},
  MRNUMBER = {91407},
MRREVIEWER = {F.\ W.\ Perkins},
       DOI = {10.1002/cpa.3160100402},
       URL = {https://doi.org/10.1002/cpa.3160100402},
}

@article {Dawson75,
    AUTHOR = {Dawson, D. A.},
     TITLE = {Stochastic evolution equations and related measure processes},
   JOURNAL = {J. Multivariate Anal.},
  FJOURNAL = {Journal of Multivariate Analysis},
    VOLUME = {5},
      YEAR = {1975},
     PAGES = {1--52},
      ISSN = {0047-259X},
   MRCLASS = {60H10},
  MRNUMBER = {388539},
MRREVIEWER = {Stanley\ L.\ Boylan},
       DOI = {10.1016/0047-259X(75)90054-8},
}

@article {JP99,
    AUTHOR = {Engl\"ander, J\'anos and Pinsky, Ross G.},
     TITLE = {On the construction and support properties of measure-valued
              diffusions on {$D\subseteq{\bf R}^d$} with spatially dependent
              branching},
   JOURNAL = {Ann. Probab.},
  FJOURNAL = {The Annals of Probability},
    VOLUME = {27},
      YEAR = {1999},
    NUMBER = {2},
     PAGES = {684--730},
      ISSN = {0091-1798,2168-894X},
   MRCLASS = {60J80 (60G57 60J60)},
  MRNUMBER = {1698955},
MRREVIEWER = {Patrick\ Fitzsimmons},
       DOI = {10.1214/aop/1022677383},
       URL = {https://doi.org/10.1214/aop/1022677383},
}

@article {JP06,
    AUTHOR = {Engl\"ander, J\'anos and Pinsky, Ross G.},
     TITLE = {The compact support property for measure-valued processes},
   JOURNAL = {Ann. Inst. H. Poincar\'e{} Probab. Statist.},
  FJOURNAL = {Annales de l'Institut Henri Poincar\'e. Probabilit\'es et
              Statistiques},
    VOLUME = {42},
      YEAR = {2006},
    NUMBER = {5},
     PAGES = {535--552},
      ISSN = {0246-0203},
   MRCLASS = {60J80 (60J60)},
  MRNUMBER = {2259973},
MRREVIEWER = {Yanxia\ Ren},
       DOI = {10.1016/j.anihpb.2005.07.001},
       URL = {https://doi.org/10.1016/j.anihpb.2005.07.001},
}

@article {Wata68,
    AUTHOR = {Watanabe, Shinzo},
     TITLE = {A limit theorem of branching processes and continuous state
              branching processes},
   JOURNAL = {J. Math. Kyoto Univ.},
  FJOURNAL = {Journal of Mathematics of Kyoto University},
    VOLUME = {8},
      YEAR = {1968},
     PAGES = {141--167},
      ISSN = {0023-608X},
   MRCLASS = {60.67},
  MRNUMBER = {237008},
MRREVIEWER = {P.\ E.\ Ney},
       DOI = {10.1215/kjm/1250524180},
}

@article{Grey_1974, title={Asymptotic behaviour of continuous time, continuous state-space branching processes}, volume={11}, DOI={10.2307/3212550}, number={4}, journal={Journal of Applied Probability}, author={Grey, D. R.}, year={1974}, pages={669–677}}

@article{SHEU1997129,
title = {Lifetime and compactness of range for super-Brownian motion with a general branching mechanism},
journal = {Stochastic Processes and their Applications},
volume = {70},
number = {1},
pages = {129-141},
year = {1997},
issn = {0304-4149},
doi = {https://doi.org/10.1016/S0304-4149(97)00059-8},
author = {Yuan-Chung Sheu},
}

@book {MR838085,
    AUTHOR = {Ethier, Stewart N. and Kurtz, Thomas G.},
     TITLE = {Markov processes: Characterization and convergence},
    SERIES = {Wiley Series in Probability and Mathematical Statistics:
              Probability and Mathematical Statistics},
 PUBLISHER = {John Wiley \& Sons, Inc.},
ADDRESS = {New York},
      YEAR = {1986},
     PAGES = {x+534},
      ISBN = {0-471-08186-8},
   MRCLASS = {60J25 (60B10 60F05 60F17 60G44 60J80)},
  MRNUMBER = {838085},
MRREVIEWER = {S.\ R. S. Varadhan},
       DOI = {10.1002/9780470316658},
}

@book {MR1280712,
    AUTHOR = {Dynkin, Eugene B.},
     TITLE = {An introduction to branching measure-valued processes},
    SERIES = {CRM Monograph Series},
    VOLUME = {6},
 PUBLISHER = {American Mathematical Society},
ADDRESS = {Providence, RI},
      YEAR = {1994},
     PAGES = {x+134},
      ISBN = {0-8218-0269-0},
   MRCLASS = {60J80 (60G57 60J25 60K99)},
  MRNUMBER = {1280712},
MRREVIEWER = {Thomas\ G.\ Kurtz},
       DOI = {10.1090/crmm/006},
}

@incollection {Dawson93,
    AUTHOR = {Dawson, Donald A.},
     TITLE = {Measure-valued {M}arkov processes},
 BOOKTITLE = {\'Ecole d'\'Et\'e{} de {P}robabilit\'es de {S}aint-{F}lour
              {XXI}---1991},
    SERIES = {Lecture Notes in Math.},
    VOLUME = {1541},
     PAGES = {1--260},
 PUBLISHER = {Springer},
ADDRESS = {Berlin},
      YEAR = {1993},
      ISBN = {3-540-56622-8},
   MRCLASS = {60G57 (60H15 60J70 60J80)},
  MRNUMBER = {1242575},
MRREVIEWER = {Luis\ G.\ Gorostiza},
       DOI = {10.1007/BFb0084190},
}

@article{dawson1994almost,
  title={Almost-sure path properties of (2, d, $\beta$)-superprocesses},
  author={Dawson, D. A. and Vinogradov, V.},
  journal={Stochastic Process. Appl.},
  fjournal={Stochastic processes and their applications},
  volume={51},
  number={2},
  pages={221--258},
  year={1994},
  publisher={Elsevier}
}

@article {DIP89,
    AUTHOR = {Dawson, D. A. and Iscoe, I. and Perkins, E. A.},
     TITLE = {Super-{B}rownian motion: path properties and hitting
              probabilities},
   JOURNAL = {Probab. Theory Related Fields},
  FJOURNAL = {Probability Theory and Related Fields},
    VOLUME = {83},
      YEAR = {1989},
    NUMBER = {1-2},
     PAGES = {135--205},
      ISSN = {0178-8051,1432-2064},
   MRCLASS = {60G17 (60G57)},
  MRNUMBER = {1012498},
MRREVIEWER = {Steven\ N.\ Evans},
       DOI = {10.1007/BF00333147},
       URL = {https://doi.org/10.1007/BF00333147},
}

@article {HK14,
    AUTHOR = {Hesse, Marion and Kyprianou, Andreas E.},
     TITLE = {The mass of super-{B}rownian motion upon exiting balls and
              {S}heu's compact support condition},
   JOURNAL = {Stochastic Process. Appl.},
  FJOURNAL = {Stochastic Processes and their Applications},
    VOLUME = {124},
      YEAR = {2014},
    NUMBER = {6},
     PAGES = {2003--2022},
      ISSN = {0304-4149,1879-209X},
   MRCLASS = {60J68 (60J80)},
  MRNUMBER = {3188347},
MRREVIEWER = {Jos\'e\ Villa-Morales},
       DOI = {10.1016/j.spa.2014.01.011},
       URL = {https://doi.org/10.1016/j.spa.2014.01.011},
}

@article {DhersinLeGall98,
    AUTHOR = {Dhersin, Jean-St\'ephane and Le Gall, Jean-Fran\c cois},
     TITLE = {Kolmogorov's test for super-{B}rownian motion},
   JOURNAL = {Ann. Probab.},
  FJOURNAL = {The Annals of Probability},
    VOLUME = {26},
      YEAR = {1998},
    NUMBER = {3},
     PAGES = {1041--1056},
      ISSN = {0091-1798,2168-894X},
   MRCLASS = {60J80 (60G17 60G57)},
  MRNUMBER = {1634414},
MRREVIEWER = {Steven\ N.\ Evans},
       DOI = {10.1214/aop/1022855744},
       URL = {https://doi.org/10.1214/aop/1022855744},
}

@article {Osserman57,
    AUTHOR = {Osserman, Robert},
     TITLE = {On the inequality {$\Delta u\geq f(u)$}},
   JOURNAL = {Pacific J. Math.},
  FJOURNAL = {Pacific Journal of Mathematics},
    VOLUME = {7},
      YEAR = {1957},
     PAGES = {1641--1647},
      ISSN = {0030-8730,1945-5844},
   MRCLASS = {35.00},
  MRNUMBER = {98239},
MRREVIEWER = {R.\ M.\ Redheffer},
       URL = {http://projecteuclid.org/euclid.pjm/1103043236},
}

@article{Dynkin1991PDE,
  title={A probabilistic approach to one class of nonlinear differential equations},
  author={E. B. Dynkin},
  journal={Probability Theory and Related Fields},
  year={1991},
  volume={89},
  pages={89-115},
}

@book{BGT87,
  author    = {Bingham, N. H. and Goldie, C. M. and Teugels, J. L.},
  title     = {Regular Variation},
  series    = {Encyclopedia of Mathematics and its Applications},
  volume    = {27},
  publisher = {Cambridge University Press},
  address   = {Cambridge},
  year      = {1987},
  doi       = {10.1017/CBO9780511721434}
}

@book{Oksendal03,
  author    = {{\O}ksendal, Bernt},
  title     = {Stochastic Differential Equations: An Introduction with Applications},
  edition   = {6},
  publisher = {Springer},
  year      = {2003}
}

@article {Iscoe86,
    AUTHOR = {Iscoe, I.},
     TITLE = {A weighted occupation time for a class of measure-valued
              branching processes},
   JOURNAL = {Probab. Theory Relat. Fields},
  FJOURNAL = {Probability Theory and Related Fields},
    VOLUME = {71},
      YEAR = {1986},
    NUMBER = {1},
     PAGES = {85--116},
      ISSN = {0178-8051,1432-2064},
   MRCLASS = {60J80},
  MRNUMBER = {814663},
MRREVIEWER = {Harry\ I.\ Cohn},
       DOI = {10.1007/BF00366274},
}

@incollection {MR1915445,
    AUTHOR = {Perkins, Edwin},
     TITLE = {Dawson-{W}atanabe superprocesses and measure-valued
              diffusions},
 BOOKTITLE = {Lectures on probability theory and statistics
              ({S}aint-{F}lour, 1999)},
    SERIES = {Lecture Notes in Math.},
    VOLUME = {1781},
     PAGES = {125--324},
 PUBLISHER = {Springer, Berlin},
      YEAR = {2002},
      ISBN = {3-540-43736-3},
   MRCLASS = {60G57 (60H15 60J80 60K35)},
  MRNUMBER = {1915445},
MRREVIEWER = {Anton\ Wakolbinger},
}

@book {LeGall99,
    AUTHOR = {Le Gall, Jean-Fran\c{c}ois},
     TITLE = {Spatial branching processes, random snakes and partial
              differential equations},
    SERIES = {Lectures in Mathematics ETH Z\"urich},
 PUBLISHER = {Birkh\"auser Verlag, Basel},
      YEAR = {1999},
     PAGES = {x+163},
      ISBN = {3-7643-6126-3},
   MRCLASS = {60J80 (35B05 35J60 35R60 60G57 60J85 60K35)},
  MRNUMBER = {1714707},
MRREVIEWER = {John\ Verzani},
       DOI = {10.1007/978-3-0348-8683-3},
       URL = {https://doi.org/10.1007/978-3-0348-8683-3},
}

@article {CD10,
    AUTHOR = {Costin, O. and Dupaigne, L.},
     TITLE = {Boundary blow-up solutions in the unit ball: asymptotics,
              uniqueness and symmetry},
   JOURNAL = {J. Differential Equations},
  FJOURNAL = {Journal of Differential Equations},
    VOLUME = {249},
      YEAR = {2010},
    NUMBER = {4},
     PAGES = {931--964},
      ISSN = {0022-0396,1090-2732},
   MRCLASS = {35J91 (35B07 35B44 35C20)},
  MRNUMBER = {2652158},
MRREVIEWER = {Dian\ K.\ Palagachev},
       DOI = {10.1016/j.jde.2010.02.023},
       URL = {https://doi.org/10.1016/j.jde.2010.02.023},
}

@article {Ren04,
    AUTHOR = {Ren, Yan-Xia},
     TITLE = {Support properties of super-{B}rownian motions with spatially
              dependent branching rate},
   JOURNAL = {Stochastic Process. Appl.},
  FJOURNAL = {Stochastic Processes and their Applications},
    VOLUME = {110},
      YEAR = {2004},
    NUMBER = {1},
     PAGES = {19--44},
      ISSN = {0304-4149,1879-209X},
   MRCLASS = {60J80 (60G57 60J45)},
  MRNUMBER = {2052135},
MRREVIEWER = {Peter\ M\"orters},
       DOI = {10.1016/j.spa.2003.09.009},
       URL = {https://doi.org/10.1016/j.spa.2003.09.009},
}

@article {Sheu94,
    AUTHOR = {Sheu, Yuan-Chung},
     TITLE = {Asymptotic behavior of superprocesses},
   JOURNAL = {Stochastics Stochastics Rep.},
  FJOURNAL = {Stochastics and Stochastics Reports},
    VOLUME = {49},
      YEAR = {1994},
    NUMBER = {3-4},
     PAGES = {239--252},
      ISSN = {1045-1129},
   MRCLASS = {60G17 (60G57 60J45 60J80)},
  MRNUMBER = {1785007},
}

@book {Li22,
    AUTHOR = {Li, Zenghu},
     TITLE = {Measure-valued branching {M}arkov processes},
    SERIES = {Probability Theory and Stochastic Modelling},
    VOLUME = {103},
   EDITION = {Second},
 PUBLISHER = {Springer},
ADDRESS = {Berlin},
      YEAR = {2022},
     PAGES = {xv+475},
      ISBN = {978-3-662-66909-9; 978-3-662-66910-5},
   MRCLASS = {60-02 (60G57 60J35 60J40 60J68 60J70 60J80 60J85)},
  MRNUMBER = {4704078},
       DOI = {10.1007/978-3-662-66910-5},
}
\end{document}